\documentclass{amsart}
\usepackage[foot]{amsaddr}
\usepackage{amssymb,amsmath,amsfonts,mathptmx,cite,mathrsfs,url}
\usepackage{eucal}

\DeclareSymbolFont{rsfscript}{OMS}{rsfs}{m}{n}
\DeclareSymbolFontAlphabet{\mathrsfs}{rsfscript}

\newcommand{\sgps}{semi\-groups}
\newcommand{\mA}{\mathcal{A}}
\newcommand{\mS}{\mathcal{S}}
\newcommand{\mG}{\mathcal{G}}
\newcommand{\mH}{\mathcal{H}}

\newcommand{\gL}{\mathrsfs{L}}
\newcommand{\gR}{\mathrsfs{R}}
\newcommand{\gH}{\mathrsfs{H}}
\newcommand{\gD}{\mathrsfs{D}}
\newcommand{\gJ}{\mathrsfs{J}}
\newcommand{\K}{\mathrsfs{K}}

\DeclareMathOperator{\alf}{alph}

\usepackage{tikz}
\usetikzlibrary{arrows}
\usetikzlibrary{arrows,shapes}

\newtheorem{theorem}{Theorem}[section]

\newtheorem{proposition}[theorem]{Proposition}
\newtheorem{lemma}[theorem]{Lemma}
\newtheorem{corollary}[theorem]{Corollary}

\theoremstyle{remark}
\newtheorem{remark}{Remark}

\newcommand{\wire}[2]{\ensuremath{#1\,{\rule[2pt]{11pt}{.9pt}}\,#2}}

\numberwithin{equation}{section}

\makeatletter

\renewcommand*\subjclass[2][2010]{\def\@subjclass{#2}\@ifundefined{subjclassname@#1}{\ClassWarning{\@classname}{Unknown edition (#1) of Mathematics Subject Classification; using '2010'.}}{\@xp\let\@xp\subjclassname\csname subjclassname@#1\endcsname}}

\renewcommand{\subjclassname}{\textup{2010} Mathematics Subject Classification}

\makeatother

\begin{document}
\title[Identities in twisted Brauer monoids]{Identities in twisted Brauer monoids}

\author[N. V. Kitov]{Nikita V. Kitov}
\address{Institute of Natural Sciences and Mathematics\\
Ural Federal University\\ 620000 Ekaterinburg, Russia}
\email{n.v.kitov@urfu.ru}
\email{m.v.volkov@urfu.ru}

\author[M. V. Volkov]{Mikhail V. Volkov}

\thanks{The authors were supported by the Ministry of Science and Higher Education of the Russian Federation, project FEUZ-2023-2022.}

\begin{abstract}
We show that it is co-NP-hard to check whether a given semigroup identity holds in the twisted Brauer monoid $\mathcal{B}^\tau_n$ with $n\ge5$.
\end{abstract}

\keywords{Twisted Brauer monoid, Identity checking problem}

\subjclass{20M07, 68Q17}

\maketitle

\section{Introduction}
\label{sec:intro}

A \emph{semigroup word} is merely a finite sequence of symbols, called \emph{letters}. An \emph{identity} is a pair of semigroup words, traditionally written as a formal equality. We write identities using the sign $\bumpeq$, so that the pair $(w,w')$ is written as $w\bumpeq w'$, and reserve the usual equality sign $=$ for `genuine' equalities. For a semigroup word $w$, the set of all letters that occur in $w$ is denoted by $\alf(w)$. If $\mathcal{S}$ is a semigroup, any map $\varphi\colon\alf(w)\to\mathcal{S}$ is called a \emph{substitution}; the element of $\mathcal{S}$ that one gets by substituting $\varphi(x)$ for each letter $x\in\alf(w)$ and computing the product in $\mathcal{S}$ is denoted by $\varphi(w)$ and called the \emph{value} of $w$ under $\varphi$.

Let $w\bumpeq w'$ be an identity, and let $X=\alf(ww')$.  We say that a semigroup $\mathcal{S}$ \emph{satisfies} $w\bumpeq w'$ (or $w\bumpeq w'$ \emph{holds} in $\mS$) if $\varphi(w)=\varphi(w')$ for every substitution $\varphi\colon X\to\mathcal{S}$, that is, each substitution of~elements in $\mathcal{S}$ for letters in $X$ yields equal values to $w$ and $w'$.

Given a semigroup $\mathcal{S}$, its \emph{identity checking problem}, denoted \textsc{Check-Id}($\mathcal{S}$), is~a combinatorial decision problem whose instance is an identity $w\bumpeq w'$; the answer to the instance $w\bumpeq w'$ is ``YES'' if $\mathcal{S}$ satisfies $w\bumpeq w'$ and~``NO'' otherwise. An alternative name for this problem that sometimes appears in the literature is the `\emph{term equivalence problem}'.

The identity checking problem is obviously decidable for finite semigroups. An active research direction aims at classifying finite semigroups $\mathcal{S}$ according to the computational complexity of \textsc{Check-Id}($\mathcal{S}$); see \cite[Section 1]{KV20} for a brief overview and references. For an infinite semigroup, the identity checking problem can be undecidable; for an example of such a semigroup, see \cite{Mu68}. On the other hand, many infinite semigroups that naturally arise in mathematics such as \sgps\ of matrices over an infinite field, or \sgps\ of relations on an infinite domain, or \sgps\ of transformations of an infinite set satisfy only identities of the form $w\bumpeq w$, and hence, the identity checking problem for such `big' \sgps\ is trivially decidable in linear time. Another family of natural infinite semigroups with linear time identity checking comes from various additive and multiplicative structures in arithmetics and commutative algebra, typical representatives being the semigroups of positive integers under addition or multiplication. It is folklore that these commutative semigroups satisfy exactly so-called balanced identities. (An identity $w\bumpeq w'$ is \emph{balanced} if every letter occurs in $w$ and $w'$ the same number of times.  Clearly, the balancedness of $w\bumpeq w'$ can be verified in linear in $|ww'|$ time.)

The latter example shows in a nutshell a common approach to identity checking in semigroups. Given a semigroup $\mathcal{S}$, one looks for a~\emph{combinatorial characterization} of~the identities holding in $\mathcal{S}$ that could be effectively verified. Recently such characterizations have been found for some infinite semigroups of interest, including, e.g., the free 2-generated semiband $\mathcal{J}_\infty=\langle e,f \mid e^2=e,\ f^2=f\rangle$ \cite{ShV17}, the bicyclic monoid $\mathcal{B}=\langle p,q \mid qp=1\rangle$ \cite{DJK18}, the Kauffman monoids $\mathcal{K}_3$ and  $\mathcal{K}_4$ \cite{Chen20,KV20}, and several monoids originated in combinatorics of tableaux such as hypoplactic, stalactic, taiga, sylvester, and Baxter monoids \cite{CMR21,HZ21,CMR22,CJKM22}. Therefore the identity checking problem in each of these semigroups is solvable in polynomial time. On the other hand, no natural examples of infinite semigroups with decidable but computationally hard identity checking seem to have been published so far. The aim of the present paper is to exhibit a series of such examples. Namely, we show that for the twisted Brauer monoid $\mathcal{B}^\tau_n$ with $n\ge5$, the problem \textsc{Check-Id}($\mathcal{B}^\tau_n$) is co-NP-hard.

The paper is structured as follows. In Sect.~\ref{sec:brauer} we first recall the definition of the twisted Brauer monoids $\mathcal{B}^\tau_n$. Then we show that for each $n$, the monoid $\mathcal{B}^\tau_n$ embeds into a regular monoid that has much better structure properties albeit it satisfies exactly the same identities as $\mathcal{B}^\tau_n$. In Sect.~\ref{sec:almeida} we modify an approach devised in~\cite{AVG09} to deal with the identity checking problem for finite semigroups so that the modified version applies to infinite semigroups subject to some finiteness conditions. In Sect.~\ref{sec:main} we prove our main result (Theorem~\ref{thm:main}), and Section~\ref{sec:conclusion} collects some additional remarks and discusses future work.

We assume the reader's acquaintance with a few basic concepts of semigroup theory, including Green's relations and presentations of semigroups via generators and relations. The first chapters of Howie's classic textbook \cite{Howie:1995} contain everything we need. For computational complexity notions, we refer the reader to Papadimitriou's textbook \cite{Pa94}.

\section{Twisted Brauer monoids}
\label{sec:brauer}

\subsection{Definition}
\label{subsec:twisted}
Twisted Brauer monoids can be defined in various ways. Here we give their geometric definition, following~\cite{ACHLV15} (where the name `wire monoids' was used).

Let $[n]=\{1,\dots,n\}$ and let $[n]'=\{1',\dots,n'\}$ be a disjoint copy of $[n]$. Consider the set $\mathcal{B}^\tau_n$ of all pairs $(\pi;s)$ where $\pi$ is a partition of the $2n$-element set $[n]\cup [n]'$ into 2-element blocks and $s$ is a nonnegative integer. Such a pair is represented by a \emph{diagram} as shown in Fig.~\ref{fig:diagram} (borrowed from~\cite{ACHLV15}).
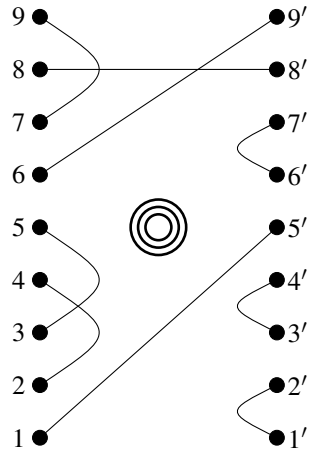
\begin{figure}[htb]
\centering
\begin{tikzpicture}
[scale=0.7]
\foreach \x in {0,4.5} \foreach \y in {0,1,2,3,4,5,6,7,8} \filldraw (\x,\y) circle (4pt);
\node[] at (-0.4,0) {$1$};
\node[] at (-0.4,1) {$2$};
\node[] at (-0.4,2) {$3$};
\node[] at (-0.4,3) {$4$};
\node[] at (-0.4,4) {$5$};
\node[] at (-0.4,5) {$6$};
\node[] at (-0.4,6) {$7$};
\node[] at (-0.4,7) {$8$};
\node[] at (-0.4,8) {$9$};
\node[] at (4.9,0) {$1'$};
\node[] at (4.9,1) {$2'$};
\node[] at (4.9,2) {$3'$};
\node[] at (4.9,3) {$4'$};
\node[] at (4.9,4) {$5'$};
\node[] at (4.9,5) {$6'$};
\node[] at (4.9,6) {$7'$};
\node[] at (4.9,7) {$8'$};
\node[] at (4.9,8) {$9'$};
    \draw (0,0) -- (4.5,4);
    \draw (0,5) -- (4.5,8);
    \draw (0,7) -- (4.5,7);
    \draw (0,1) .. controls (1.5,2) ..  (0,3);
    \draw (0,2) .. controls (1.5,3) ..  (0,4);
    \draw (0,6) .. controls (1.5,7) ..  (0,8);	
    \draw (4.5,0) .. controls (3.5,0.5) ..  (4.5,1);
    \draw (4.5,2) .. controls (3.5,2.5) ..  (4.5,3);
    \draw (4.5,5) .. controls (3.5,5.5) ..  (4.5,6);
\foreach \x in {2.25} \foreach \y in {4} \draw [line width=1pt] (\x,\y) circle (7pt);
\foreach \x in {2.25} \foreach \y in {4} \draw [line width=1pt] (\x,\y) circle (11pt);
\foreach \x in {2.25} \foreach \y in {4} \draw [line width=1pt] (\x,\y) circle (15pt);
\end{tikzpicture}
\caption{Diagram representing an element of $\mathcal{B}^\tau_9$}\label{fig:diagram}
\end{figure}
We represent the elements of $[n]$ by points on the left-hand side of the diagram (\emph{left points}) while the elements of $[n]'$ are represented by points on the right-hand side of the diagram (\emph{right points}). For $(\pi;s)\in \mathcal{B}^\tau_n$, we represent the number $s$ by $s$ closed curves (called \emph{circles} or \emph{floating components}) and each block of the partition $\pi$ is represented by a line referred to as a \emph{wire}. Thus, each wire connects two points; it is called an $\ell$-\emph{wire} if it connects two left points, an $r$-\emph{wire} if it connects two right points, and a $t$-\emph{wire} if it connects a left point with a right point. The  diagram in Fig.~\ref{fig:diagram} has three wires of each type and three circles; it corresponds to the pair
$$\Bigl(\bigr\{\{1,5'\},\{2,4\},\{3,5\},\{6,9'\},\{7,9\},\{8,8'\},\{1',2'\},\{3',4'\},\{6',7'\}\bigr\};\,3\Bigr).$$

In what follows we use `vertical' diagrams as the one in Fig.~\ref{fig:diagram} but in the literature (see, e.g., \cite{DE18}) the reader can also meet representations of pairs from $\mathcal{B}^\tau_n$ by `horizontal' diagrams like the one in Fig.~\ref{fig:horizontal diagram}.
\begin{figure}[ht]
\centering
\begin{tikzpicture}
\foreach \x in {1,2,...,8} \draw (\x,0) node[above=-1pt] {$\x$};
\foreach \x in {1,2,...,8} \draw (\x,-2) node[below=-1pt] {$\x'$};
\draw[thick] (1,0)--(1,-2); \draw[thick] (2,0)--(3,-2); \draw[thick](4,0)--(5,-2);\draw[thick](7,0)--(8,-2);
\draw[thick] (2,-2)..controls(3,-1.25)..(4,-2); \draw[thick] (3,0).. controls (4,-0.75)..(5,0);\draw[thick] (6,-2)..controls (6.5,-1.5)..(7,-2); \draw[thick](6,0).. controls(7,-0.75)..(8,0);
\draw[thick] (10,-1.4) circle(0.5cm);
\draw[thick] (11,-0.9) circle(1cm);
\draw[thick] (12,-0.9) circle(0.5cm);
\draw[thick] (9,-0.5) circle(0.5cm);
\end{tikzpicture}
\caption{Diagram representing an element of $\mathcal{B}^\tau_8$}\label{fig:horizontal diagram}
\end{figure}
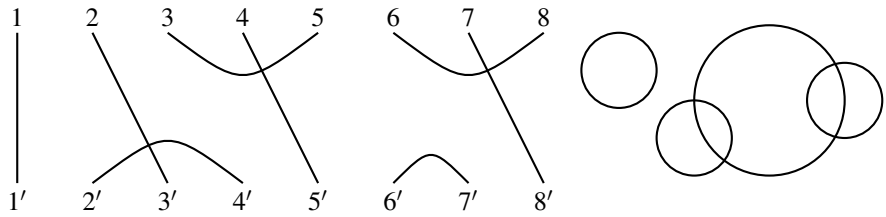
Of course, the `vertical' and `horizontal' viewpoints are fully equivalent. We also stress that only two things matter in any diagrammatic representation of the elements of $\mathcal{B}^\tau_n$, namely, 1) which points are connected and 2) the number of circles; neither the shape nor the relative position of the wires and circles matters. For instance, the diagram in Fig.~\ref{fig:another diagram} represents the same element of $\mathcal{B}^\tau_8$ as the diagram in Fig.~\ref{fig:horizontal diagram}.
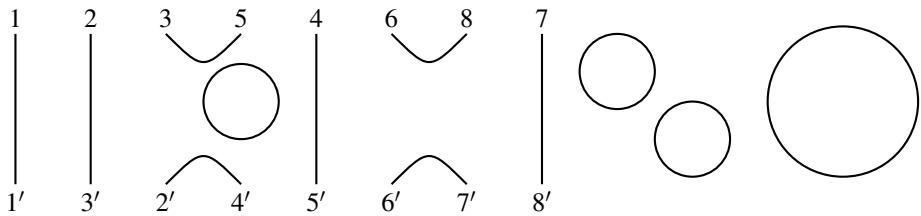
\begin{figure}[hb]
\centering
\begin{tikzpicture}
\draw (1,0) node[above=-1pt] {$1$};
\draw (2,0) node[above=-1pt] {$2$};
\draw (3,0) node[above=-1pt] {$3$};
\draw (4,0) node[above=-1pt] {$5$};
\draw (5,0) node[above=-1pt] {$4$};
\draw (6,0) node[above=-1pt] {$6$};
\draw (7,0) node[above=-1pt] {$8$};
\draw (8,0) node[above=-1pt] {$7$};
\draw (1,-2) node[below=-1pt] {$1'$};
\draw (2,-2) node[below=-1pt] {$3'$};
\draw (3,-2) node[below=-1pt] {$2'$};
\draw (4,-2) node[below=-1pt] {$4'$};
\draw (5,-2) node[below=-1pt] {$5'$};
\draw (6,-2) node[below=-1pt] {$6'$};
\draw (7,-2) node[below=-1pt] {$7'$};
\draw (8,-2) node[below=-1pt] {$8'$};
\draw[thick] (1,0)--(1,-2);
\draw[thick] (2,0)--(2,-2);
\draw[thick] (5,0)--(5,-2);
\draw[thick] (8,0)--(8,-2);
\draw[thick] (3,0)..controls(3.5,-0.5)..(4,0);
\draw[thick] (3,-2)..controls(3.5,-1.5)..(4,-2);
\draw[thick] (6,0)..controls(6.5,-0.5)..(7,0);
\draw[thick] (6,-2)..controls(6.5,-1.5)..(7,-2);
\draw[thick] (10,-1.4) circle(0.5cm);
\draw[thick] (12,-0.9) circle(1cm);
\draw[thick] (4,-0.9) circle(0.5cm);
\draw[thick] (9,-0.5) circle(0.5cm);
\end{tikzpicture}
\caption{Diagram of Fig.~\ref{fig:horizontal diagram} redrawn}\label{fig:another diagram}
\end{figure}

Now we define a multiplication in $\mathcal{B}^\tau_n$. Pictorially, in order to multiply two diagrams, we glue their wires together by identifying each right point $u'$ of the first diagram with the corresponding left point $u$ of the second diagram. This way we obtain a new diagram whose left (respectively, right) points are the left (respectively, right) points of the first (respectively, second) diagram. Two points of this new diagram are connected in it if one can reach one of them from the other by walking along a sequence of consecutive wires of the factors, see Fig.~\ref{fig:multiplication} (where the labels $1,2,\dots,9,1',2',\dots,9'$ are omitted but they are assumed to go up in the consecutive order). All circles of the factors are inherited by the product; in addition, some extra circles may arise from $r$-wires of the first diagram combined with $\ell$-wires of the second diagram.

In more precise terms, if $\xi=(\pi_1;s_1)$, $\eta=(\pi_2;s_2)$, then a left point $p$ and a right point $q'$ of the product $\xi\eta$ are connected by a $t$-wire if and only if one of the following holds:
\begin{trivlist}
\item[$\bullet$] \wire{p}{u'} is a $t$-wire in $\xi$ and \wire{u}{q'} is a $t$-wire in $\eta$ for some $u\in[n]$;
\item[$\bullet$] for some $s>1$ and some $u_1,v_1,u_2,\dots,v_{s-1},u_s\in[n]$ (all pairwise distinct), \wire{p}{u_1'} is a $t$-wire in $\xi$ and \wire{u_s}{q'} is a $t$-wire in $\eta$, while \wire{u_i}{v_i} is an $\ell$-wire in $\eta$ and \wire{v_i'}{u_{i+1}'} is an $r$-wire in $\xi$ for each $i=1,\dots,s-1$.\\
(The reader may trace an application of the second rule in Fig.~\ref{fig:multiplication}, in which such a `composite' $t$-wire connects 1 and $3'$ in the product diagram.)
\end{trivlist}
\begin{figure}[ht]
\centering
\begin{tikzpicture}
[scale=0.55]
\foreach \x in {0,4.5} \foreach \y in {0,1,2,3,4,5,6,7,8} \filldraw (\x,\y) circle (4pt);
    \draw (0,0) -- (4.5,4);
    \draw (0,5) -- (4.5,8);
    \draw (0,7) -- (4.5,7);
    \draw (0,1) .. controls (1.5,2) ..  (0,3);
    \draw (0,2) .. controls (1.5,3) ..  (0,4);
    \draw (0,6) .. controls (1.5,7) ..  (0,8);	
    \draw (4.5,0) .. controls (3.5,0.5) ..  (4.5,1);
    \draw (4.5,2) .. controls (3.5,2.5) ..  (4.5,3);
    \draw (4.5,5) .. controls (3.5,5.5) ..  (4.5,6);
\foreach \x in {2.25} \foreach \y in {4} \draw [line width=1pt] (\x,\y) circle (7pt);
\foreach \x in {2.25} \foreach \y in {4} \draw [line width=1pt] (\x,\y) circle (11pt);
\foreach \x in {2.25} \foreach \y in {4} \draw [line width=1pt] (\x,\y) circle (15pt);
\node[] at (5.5,4) {$\times$};
\foreach \x in {6.5,11} \foreach \y in {0,1,2,3,4,5,6,7,8} \filldraw (\x,\y) circle (4pt);
    \draw[red] (6.5,6) -- (11,2);
    \draw[red]  (6.5,0) .. controls (7.5,0.5) ..  (6.5,1);
    \draw[red]  (6.5,2) .. controls (8.5,3.5) ..  (6.5,5);
    \draw[red]  (6.5,3) .. controls (7.5,3.5) ..  (6.5,4);
    \draw[red]  (6.5,7) .. controls (7.5,7.5) ..  (6.5,8);	
    \draw[red]  (11,0) .. controls (10,0.5) ..  (11,1);
    \draw[red]  (11,3) .. controls (9,5.5) ..  (11,8);
    \draw[red]  (11,4) .. controls (10,4.5) ..  (11,5);
    \draw[red]  (11,6) .. controls (10,6.5) ..  (11,7);
\foreach \x in {8.75} \foreach \y in {2} \draw [red, line width=1pt] (\x,\y) circle (7pt);
\node[] at (12,4) {$=$};
\foreach \x in {13,17.5} \foreach \y in {0,1,2,3,4,5,6,7,8} \filldraw (\x,\y) circle (4pt);
    \draw (13,0) -- (17.5,4);
    \draw (13,5) -- (17.5,8);
    \draw (13,7) -- (17.5,7);
    \draw (13,1) .. controls (14.5,2) ..  (13,3);
    \draw (13,2) .. controls (14.5,3) ..  (13,4);
    \draw (13,6) .. controls (14.5,7) ..  (13,8);	
    \draw (17.5,0) .. controls (16.5,0.5) ..  (17.5,1);
    \draw (17.5,2) .. controls (16.5,2.5) ..  (17.5,3);
    \draw (17.5,5) .. controls (16.5,5.5) ..  (17.5,6);
\foreach \x in {15.25} \foreach \y in {4} \draw [line width=1pt] (\x,\y) circle (7pt);
\foreach \x in {15.25} \foreach \y in {4} \draw [line width=1pt] (\x,\y) circle (11pt);
\foreach \x in {15.25} \foreach \y in {4} \draw [line width=1pt] (\x,\y) circle (15pt);
\draw[dashed] (17.5,-0.5) -- (17.5,8.5);
\foreach \x in {17.5,22} \foreach \y in {0,1,2,3,4,5,6,7,8} \filldraw (\x,\y) circle (4pt);
    \draw[red] (17.5,6) -- (22,2);
    \draw[red]  (17.5,0) .. controls (18.5,0.5) ..  (17.5,1);
    \draw[red]  (17.5,2) .. controls (19.5,3.5) ..  (17.5,5);
    \draw[red]  (17.5,3) .. controls (18.5,3.5) ..  (17.5,4);
    \draw[red]  (17.5,7) .. controls (18.5,7.5) ..  (17.5,8);	
    \draw[red]  (22,0) .. controls (21,0.5) ..  (22,1);
    \draw[red]  (22,3) .. controls (20,5.5) ..  (22,8);
    \draw[red]  (22,4) .. controls (21,4.5) ..  (22,5);
    \draw[red]  (22,6) .. controls (21,6.5) ..  (22,7);
\foreach \x in {19.75} \foreach \y in {2} \draw [red, line width=1pt] (\x,\y) circle (7pt);
\end{tikzpicture}
\caption{Multiplication of diagrams}\label{fig:multiplication}
\end{figure}
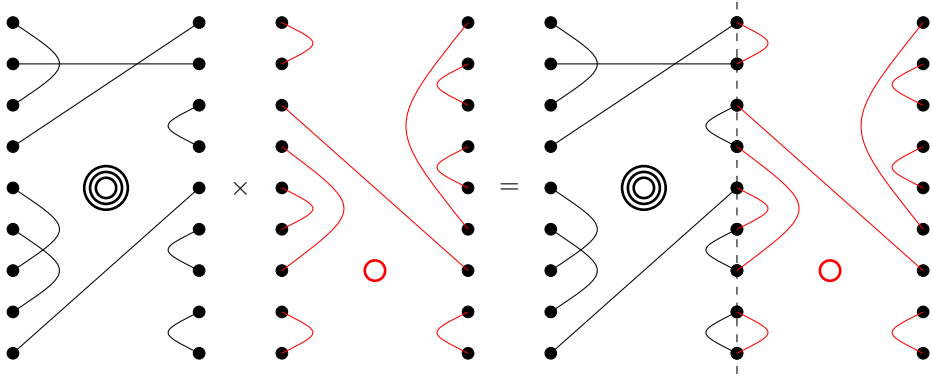

Analogous characterizations hold for the $\ell$-wires and $r$-wires of $\xi\eta$. Here we include only the rules for forming the $\ell$-wires as the $r$-wires of the product are obtained in a perfectly symmetric way.

Two left points $p$ and $q$ of $\xi\eta$ are connected by an $\ell$-wire if and only if one of the following holds:
\begin{trivlist}
\item[$\bullet$] \wire{p}{q} is an $\ell$-wire in $\xi$;
\item[$\bullet$] for some $s\ge1$ and some $u_1,v_1,u_2,\dots,v_s\in[n]$ (all pairwise distinct), \wire{p}{u_1'} and \wire{q}{v_s'} are $t$-wires in $\xi$, while \wire{u_i}{v_i} is an $\ell$-wire in $\eta$ for each $i=1,\dots,s$ and if $s>1$, then \wire{v_i'}{u_{i+1}'} is an $r$-wire in $\xi$ for each $i=1,\dots,s-1$.\\
(Again, Fig.~\ref{fig:multiplication} provides an instance of the second rule: look at the $\ell$-wire that connects 6 and 8 in the product diagram.)
\end{trivlist}

Finally, each circle of the product $\xi\eta$ corresponds to either a circle in $\xi$ or $\eta$ or a sequence $u_1,v_1,\dots,u_s,v_s\in[n]$ with $s\ge 1$ and pairwise distinct $u_1,v_1,\dots,u_s,v_s$ such that all \wire{u_i}{v_i} are $\ell$-wires in $\eta$, while all \wire{v_i'}{u_{i+1}'} and \wire{v_s'}{u_1'} are $r$-wires in $\xi$.\\ (In Fig.~\ref{fig:multiplication}, one sees such a `new' circle formed by the $\ell$-wire \wire{1}{2} of the second factor glued to the $r$-wire \wire{2'}{1'} of the first factor.)

The above defined multiplication in $\mathcal{B}^\tau_n$ is easily seen to be associative and the diagram with 0 circles and the $n$ horizontal $t$-wires \wire{1}{1'}, \dots, \wire{n}{n'} is the identity element with respect to the multiplication. Thus, $\mathcal{B}^\tau_n$ is a monoid called the \emph{twisted Brauer monoid}.

\subsection{Background}
\label{subsec:background}
We refer to \cite{DE18} for a throughout analysis of the semigroup-theoretic properties of twisted Brauer monoids. Here we explain the terminology and relate the monoids $\mathcal{B}^\tau_n$ to the representation theory of classical groups. Some parts of our discussion involve concepts from beyond semigroup theory. These parts are not used in subsequent proofs so that the reader who is only interested in our main result can safely skip them.

Let $R$ be a commutative ring with 1 and $\mathcal{S}$ a semigroup. Following \cite{Wil07}, we say that a map $\tau\colon\mathcal{S}\times\mathcal{S}\to R$ is a \emph{twisting} from $\mathcal{S}$ to $R$ if
\begin{equation}
\label{eq:twisting}
\tau(s, t) \tau(s t, u)=\tau(s, t u) \tau(t, u) \text { for all } s, t, u \in \mathcal{S}.
\end{equation}
The \emph{twisted semigroup algebra} of $\mathcal{S}$ over $R$, with twisting $\tau$, denoted by $R^\tau[\mathcal{S}]$, is the free $R$-module spanned by $\mathcal{S}$ as a basis with multiplication $\circ$ defined by
\begin{equation}
\label{eq:basis}
s \circ t=\tau(s, t)st \quad \text { for all } s, t \in \mathcal{S},
\end{equation}
and extended by linearity. Condition \eqref{eq:twisting} readily implies that this multiplication is associative. If $\tau(s, t)=1$ for all $s, t \in \mathcal{S}$, then the twisted semigroup algebra $R^\tau[\mathcal{S}]$ is nothing but the usual semigroup algebra $R[\mathcal{S}]$. Thus, twisted semigroup algebras provide a vast generalization of semigroup algebras while retaining many useful properties of the latter, in particular, those important for representation theory.

Having clarified the meaning of `twisted', let us explain what the Brauer monoid is. Denote by $\mathcal{B}_n$ the set of all partitions of the $2n$-element set $[n]\cup [n]'$ into 2-element blocks; we visualize them as diagrams similar to the one in Fig.~\ref{fig:diagram} but without floating components. For instance, Fig.~\ref{fig:B3} shows the 15 diagrams representing the partitions from $\mathcal{B}_3$.

\medskip

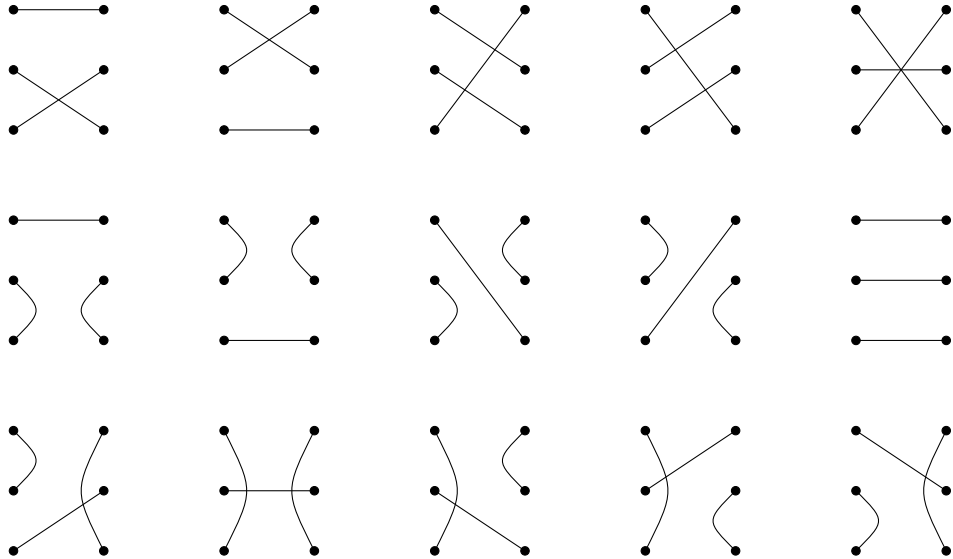
\begin{figure}[htb]
\centering
\begin{tikzpicture}
[scale=0.8]
\foreach \x in {0,1.5,3.5,5,7,8.5,10.5,12,14,15.5} \foreach \y in {0.5,1.5,2.5,-1,-2,-3,-4.5,-5.5,-6.5} \filldraw (\x,\y) circle (2pt);
	\draw (0,0.5) -- (1.5,1.5);
	\draw (0,1.5) -- (1.5,0.5);
	\draw (0,2.5) -- (1.5,2.5);
	\draw (3.5,0.5) -- (5,0.5);
	\draw (3.5,1.5) -- (5,2.5);
	\draw (3.5,2.5) -- (5,1.5);
	\draw (7,0.5) -- (8.5,2.5);
	\draw (7,1.5) -- (8.5,0.5);
	\draw (7,2.5) -- (8.5,1.5);
	\draw (10.5,0.5) -- (12,1.5);
	\draw (10.5,1.5) -- (12,2.5);
	\draw (10.5,2.5) -- (12,0.5);
	\draw (14,0.5) -- (15.5,2.5);
	\draw (14,1.5) -- (15.5,1.5);
	\draw (14,2.5) -- (15.5,0.5);
	\draw (0,-1) -- (1.5,-1);
	\draw (0,-2) .. controls (.5,-2.5) ..  (0,-3);
	\draw (1.5,-2) .. controls (1,-2.5) ..  (1.5,-3);
	\draw (3.5,-3) -- (5,-3);
	\draw (3.5,-2) .. controls (4,-1.5) ..  (3.5,-1);
	\draw (5,-2) .. controls (4.5,-1.5) ..  (5,-1);
	\draw (7,-1) -- (8.5,-3);
	\draw (7,-2) .. controls (7.5,-2.5) ..  (7,-3);
	\draw (8.5,-1) .. controls (8,-1.5) ..  (8.5,-2);
	\draw (10.5,-3) -- (12,-1);
	\draw (10.5,-2) .. controls (11,-1.5) ..  (10.5,-1);
	\draw (12,-2) .. controls (11.5,-2.5) ..  (12,-3);
    \draw (14,-1) -- (15.5,-1);
	\draw (14,-2) -- (15.5,-2);
	\draw (14,-3) -- (15.5,-3);	
	\draw (0,-6.5) -- (1.5,-5.5);
	\draw (0,-5.5) .. controls (0.5,-5) ..  (0,-4.5);
	\draw (1.5,-6.5) .. controls (1,-5.5) ..  (1.5,-4.5);
    \draw (5,-5.5) -- (3.5,-5.5);
	\draw (5,-4.5) .. controls (4.5,-5.5) ..  (5,-6.5);
	\draw (3.5,-4.5) .. controls (4,-5.5) ..  (3.5,-6.5);
	\draw (7,-5.5) -- (8.5,-6.5);
	\draw (7,-6.5) .. controls (7.5,-5.5) ..  (7,-4.5);
	\draw (8.5,-5.5) .. controls (8,-5) ..  (8.5,-4.5);
	\draw (10.5,-5.5) -- (12,-4.5);
	\draw (10.5,-6.5) .. controls (11,-5.5) ..  (10.5,-4.5);
	\draw (12,-5.5) .. controls (11.5,-6) ..  (12,-6.5);
    \draw (14,-4.5) -- (15.5,-5.5);
	\draw (14,-5.5) .. controls (14.5,-6) ..  (14,-6.5);
	\draw (15.5,-4.5) .. controls (15,-5.5) ..  (15.5,-6.5);
\end{tikzpicture}
\caption{Diagrams representing the elements of $\mathcal{B}_3$}\label{fig:B3}
\end{figure}

Identifying each partition $\pi\in\mathcal{B}_n$ with the pair $(\pi;0)\in\mathcal{B}^\tau_n$, we may treat $\mathcal{B}_n$ as a subset of $\mathcal{B}^\tau_n$. The multiplication in $\mathcal{B}^\tau_n$ induces a multiplication in $\mathcal{B}_n$ as follows: given two partitions $\pi_1,\pi_2\in\mathcal{B}_n$, one computes the product of the pairs $(\pi_1;0)$ and $(\pi_2;0)$ in $\mathcal{B}^\tau_n$ and if $(\pi_1;0)(\pi_2;0)=(\pi;s)$, one lets $\pi$ be the product of $\pi_1$ and $\pi_2$ in $\mathcal{B}_n$. In other words, one multiplies the diagrams of $\pi_1$ and $\pi_2$, using the multiplication rules in Section~\ref{subsec:twisted}, and then discards all `new' circles if they arise.

Under the above defined multiplication, $\mathcal{B}_n$ constitutes a monoid called the \emph{Brauer monoid}. This family of monoids was invented by Brauer~\cite{Br37} back in 1937, hence the name.
Brauer used the monoid $\mathcal{B}_n$ to study the linear representations of orthogonal and symplectic groups.

Observe that $\mathcal{B}_n$ is \textbf{not} a submonoid of $\mathcal{B}^\tau_n$; at the same time, $\mathcal{B}_n$ is easily seen to be the homomorphic image of $\mathcal{B}^\tau_n$ under the `forgetting' homomorphism $(\pi;s)\mapsto\pi$ that removes the circles from the diagrams in $\mathcal{B}^\tau_n$.

Given two partitions $\pi_1,\pi_2\in\mathcal{B}_n$, we denote by $\langle\pi_1,\pi_2\rangle$ the number of cycles that arise when the pairs $(\pi_1;0)$ and $(\pi_2;0)$ are multiplied in $\mathcal{B}^\tau_n$. Using this notation, we have the following useful formula expressing the multiplication in $\mathcal{B}^\tau_n$ via that in $\mathcal{B}_n$:
\begin{equation}\label{eq:mult}
(\pi_1;s_1)(\pi_2;s_2)=(\pi_1\pi_2;s_1+s_2+\langle\pi_1,\pi_2\rangle),
\end{equation}
where the product $\pi_1\pi_2$ in the right-hand side is computed in $\mathcal{B}_n$.

Now let $F$ be a field of characteristic 0. Fix an element $\theta\in F\setminus\{0\}$ and consider the map $\tau\colon\mathcal{B}_n\times\mathcal{B}_n\to F$ defined by
\[
\tau(\pi_1,\pi_2)=\theta^{\langle\pi_1,\pi_2\rangle}.
\]
It is known (and easy to verify) that $\tau$ satisfies \eqref{eq:twisting} so the map is a twisting from $\mathcal{B}_n$ to $F$. Hence, one can construct the twisted semigroup algebra $F^\tau[\mathcal{B}_n]$. In the literature, the notation and the name for this algebra  vary; we denote it by $B_n(\theta)$ as in \cite{Kerov:87} and call it \emph{Brauer's centralizer algebra} as, e.g., in~\cite{HW89}. In~\cite{Br37}, Brauer's centralizer algebra $B_n(m)$, where $m$ is a positive integer, was used to study the natural representation of the orthogonal group $\mathrm{O}_m$  on the $n$-th tensor power $(F^m)^{\otimes n}$ of the space $F^m$. (The algebra $B_n(m)$ is exactly the centralizer of the diagonal action of $\mathrm{O}_m$ on $(F^m)^{\otimes n}$, hence the name.) When $m$ is an even positive integer, Brauer's centralizer algebra $B_n(-m)$ allowed for a similar study of the representation of the symplectic group $\mathrm{Sp}_m$ on $(F^m)^{\otimes n}$. It was also present in~\cite{Br37}, albeit implicitly; see~\cite{HW89} for a detailed analysis. For the case where the parameter $\theta$ is arbitrary, the ring-theoretic structure of the algebra $B_n(\theta)$ has been determined by Wenzl~\cite{Wen88} (in particular, he has proved that if $\theta$ is not an integer, then the algebra $B_n(\theta)$ is semisimple).

In a simplified form, the approach used in~\cite{Wen88} to give a uniform treatment of the algebras $B_n(\theta)$ for various $\theta$ can be stated as follows. Define a twisting from $\mathcal{B}_n$ to the polynomial ring $F[X]$ by\footnote{Wenzl~\cite{Wen88} employed the same twisting but to the field of rational functions rather than the polynomial ring.}
\begin{equation}\label{eq:xtwist}
\tau(\pi_1,\pi_2)=X^{\langle\pi_1,\pi_2\rangle}.
\end{equation}
Then one gets the twisted semigroup algebra $(F[X])^\tau[\mathcal{B}_n]$ that can be denoted by $B_n(X)$. For each $\theta\in F\setminus\{0\}$, evaluating $X$ at $\theta$ gives rise to a homomorphism from $B_n(X)$ onto $B_n(\theta)$ so that the algebra $B_n(X)$ is a kind of mother of all Brauer's centralizer algebras.

On the other hand, $B_n(X)$ is nothing but the usual (non-twisted) semigroup algebra $F[\mathcal{B}^\tau_n]$. Indeed, it is easy to show that the bijection $\beta\colon(\pi;s)\mapsto X^s\pi$ is a semigroup isomorphism between the $F$-basis $\mathcal{B}^\tau_n$ of $F[\mathcal{B}^\tau_n]$ and the $F$-basis $\{X^s\pi\mid s\in\mathbb{Z}_{\ge0},\ \pi\in\mathcal{B}_n\}$ of $B_n(X)$ where the latter basis is equipped with multiplication $\circ$ as in  \eqref{eq:basis}:
\begin{align*}
 \beta\bigl((\pi_1;s_1)(\pi_2;s_2)\bigr) &=\beta\bigl((\pi_1\pi_2;s_1+s_2+\langle\pi_1,\pi_2\rangle)\bigr) &&\text{by \eqref{eq:mult}}  \\
                                         &=X^{s_1+s_2+\langle\pi_1,\pi_2\rangle}\pi_1\pi_2 &&\text{by the definition of $\beta$} \\
                                         &=X^{s_1}\pi_1\circ X^{s_2}\pi_2 && \text{by the definition of $\circ$; see \eqref{eq:basis}}\\
                                         &=\beta(\pi_1;s_1)\circ\beta(\pi_2;s_2)&&\text{by the definition of $\beta$.}
\end{align*}
The isomorphism $(F[X])^\tau[\mathcal{B}_n]\cong F[\mathcal{B}^\tau_n]$, which moves the twist from the outer ring to the inner semigroup, stands behind our terminology: instead of twisting the semigroup algebra of the Brauer monoid, we twist the monoid itself, thus getting the twisted Brauer monoid.

\subsection{Presentation for $\mathcal{B}^\tau_n$}
\label{subsec:presentation}
It is known that the twisted Brauer monoid $\mathcal{B}^\tau_n$ can be generated by the following $2n-1$ pairs:
\begin{itemize}
  \item \emph{transpositions} $t_i=\Bigl(\bigr\{ \wire{i}{(i+1)'},\ \wire{i'}{i+1},\  \wire{j}{j'} \mid \text{for } j\ne i,i+1\bigr\};\,0\Bigr)$,\\ $i=1,\dots,n-1$,
  \item \emph{hooks} $h_i=\Bigl(\bigr\{\wire{i}{i+1},\ \wire{i'}{(i+1)'},\  \wire{j}{j'} \mid \text{for } j\ne i,i+1\bigr\};\,0\Bigr)$,\\ $i=1,\dots,n-1$,
  \item and the \emph{circle} $c=\Bigl(\bigr\{\wire{j}{j'} \mid \text{for } j=1,\dots,n\bigr\};\,1\Bigr).$
\end{itemize}
For an illustration, see Fig.~\ref{fig:B3}: the first two diagrams in the top row represent the transpositions $t_1$ and $t_2$, and the first two diagrams in the middle row represent the hooks $h_1$ and $h_2$. (The omitted labels $1,2,3,1',2',3'$ are assumed to go up in the consecutive order.)

For all $i,j=1,\dots,n-1$, the generators $t_1,\dots,t_{n-1},h_1,\dots,h_{n-1},c$ satisfy the following:
\begin{align}
&t_{i}^{2}=1,                    &&\label{eq:M1}\\
&t_{i}t_{j}=t_{j}t_{i}           &&\text{if } |i-j|\ge 2,\label{eq:M2}\\
&t_{i}t_{j}t_{i}=t_{j}t_{i}t_{j} &&\text{if } |i-j|=1,\label{eq:M3}\\
&ct_{i}=t_{i}c,                  &&\label{eq:C}\\
&h_{i}h_{j}=h_{j}h_{i}           &&\text{if } |i-j|\ge 2,\label{eq:TL1}\\
&h_{i}h_{j}h_{i}=h_{i}           &&\text{if } |i-j|=1,\label{eq:TL2}\\
&h_{i}^2=ch_{i}=h_{i}c,          &&\label{eq:TL3}\\
&h_{i}t_{i}=t_{i}h_{i}=h_{i},    &&\label{eq:Mix1}\\
&h_{i}t_{j}=t_{j}h_{i}           &&\text{if } |i-j|\ge2,\label{eq:Mix2}\\
&t_{i}h_{j}h_{i}=t_{j}h_{i},\ h_{i}h_{j}t_{i}=h_{i}t_{j} &&\text{if } |i-j|=1.\label{eq:Mix3}
\end{align}

\begin{proposition}
\label{prop:presentation}
The relations \eqref{eq:M1}--\eqref{eq:Mix3} is a monoid presentation of the twisted Brauer monoid $\mathcal{B}^\tau_n$ with respect to the generators $t_1,\dots,t_{n-1},h_1,\dots,h_{n-1},c$.
\end{proposition}

Even though Proposition~\ref{prop:presentation} does not seem to have been registered in the literature, its result is not essentially new as it is an immediate combination of two known ingredients: 1) a presentation for the
Brauer monoid $\mathcal{B}_n$~\cite[Section 3]{KM06}, and 2) a method for `twisting' presentations, that is, obtaining a presentation for a twisted semigroup algebra from a given presentation of the underlying semigroup of the algebra~\cite[Section 6 and Remark 45]{Ea11}\footnote{The authors are grateful to Dr. James East who drew their attention to this combination.}.

The presentation of Proposition~\ref{prop:presentation} is not the most economical one in terms of the number of generators (in fact, for each $n$, the monoid $\mathcal{B}^\tau_n$ can be generated by just four elements, see \cite[Proposition 3.11]{DE18}). It is, however, quite transparent and conveniently reveals some structural components of $\mathcal{B}^\tau_n$.

For instance, the relations \eqref{eq:M1}--\eqref{eq:M3} are nothing but Moore's classical relations \cite[Theorem A]{Moore97} for the symmetric group $\mathbb{S}_n$. Since both sides of any other relation involve some generator beside the transpositions $t_1,\dots,t_{n-1}$, no other relation can be applied to a word composed of transpositions only. Therefore, the transpositions generate in $\mathcal{B}^\tau_n$ a subgroup isomorphic to $\mathbb{S}_n$; this subgroup is actually the group of units of $\mathcal{B}^\tau_n$. The reader may use Fig.~\ref{fig:B3} as an illustration for $n=3$: the five diagrams on the top row, together with the last diagram of the middle row, represent the elements of $\mathbb{S}_3$.

On the other hand, the relations \eqref{eq:TL1}--\eqref{eq:TL3} that do not involve $t_1,\dots,t_{n-1}$ are the so-called Temperley--Lieb relations constituting a presentation for the \emph{Kauffman monoid} $\mathcal{K}_n$.
(We mentioned the monoids $\mathcal{K}_3$ and  $\mathcal{K}_4$ from this family in the introduction.) Again, Fig.~\ref{fig:B3} provides an illustration for $n=3$: the five diagrams of the middle row represent the partitions $\pi$ such that $(\pi;s)\in \mathcal{K}_3$ for any $s\in\mathbb{Z}_{\ge 0}$. One sees that these five diagrams are exactly those whose wires do not cross. It is this property that Kauffman \cite{Ka90} used to introduce the monoids $\mathcal{K}_n$; namely, he defined $\mathcal{K}_n$ as the submonoid of $\mathcal{B}^\tau_n$ consisting of all elements of $\mathcal{B}^\tau_n$ that have a representation as a diagram in which the labels $1,2,\dots,n,1',2',\dots,n'$ go up in the consecutive order and wires do not cross. The fact that the submonoid can be identified with the monoid generated by $h_1,\dots,h_{n-1},c$ subject to the relations  \eqref{eq:TL1}--\eqref{eq:TL3} was stated in~\cite{Ka90} with a proof sketch; for a detailed proof, see \cite{BDP02} or~\cite{Ea21}.

\subsection{The monoid $\mathcal{B}^{\pm\tau}_n$ and its identities}
\label{subsec:embedding}

Following an idea by Karl Auinger (personal communication), we embed the monoid $\mathcal{B}^\tau_n$ into a larger monoid $\mathcal{B}^{\pm\tau}_n$ that shares the identities with $\mathcal{B}^\tau_n$ but has much better structure properties. In~\cite{KV20}, we applied the same trick to Kauffman monoids.

In terms of generators and relations, the $\pm$-\emph{twisted Brauer monoid} $\mathcal{B}^{\pm\tau}_n$ can be defined as the monoid with $2n$ generators $t_1,\dots,t_{n-1},h_1,\dots,h_{n-1},c,d$ subject to the relations \eqref{eq:M1}--\eqref{eq:Mix3} and the additional relations
\begin{equation}
\label{eq:inverse}
cd=dc=1.
\end{equation}
Observe that \eqref{eq:C} and \eqref{eq:inverse} imply that $dt_i=t_id$ for each $i=1,\dots,n-1$. Indeed,
\begin{align*}
dt_i&=d^2ct_i&&\text{since $dc=1$}\\
    &=d^2t_ic&&\text{since $ct_i=t_ic$}\\
    &=d^2t_ic^2d&&\text{since $cd=1$}\\
    &=d^2c^2t_id&&\text{since $c^2t_i=t_ic^2$}\\
    &=t_id&&\text{since $d^2c^2=1$.}
\end{align*}
Similarly, the  relations \eqref{eq:TL3} and \eqref{eq:inverse} imply that $dh_i=h_id$ for each $i=1,\dots,n-1$.

It is easy to see that the submonoid of $\mathcal{B}^{\pm\tau}_n$ generated by $t_1,\dots,t_{n-1},h_1,\dots,h_{n-1},c$ is isomorphic to $\mathcal{B}^\tau_n$. The generator $d$ commutes with every element of this submonoid since is commutes with each of its generators.

To interpret the $\pm$-twisted Brauer monoid in terms of diagrams, we introduce two sorts of circles: positive and negative. Each diagram may contain only circles of one sort. When two diagrams are multiplied, the following two rules are obeyed: all `new' circles (the ones that arise when the diagrams are glued together) are positive; in addition, if the product diagram inherits some negative circles from its factors, then pairs of `opposite' circles are consecutively removed until only circles of a single sort (or no circles at all) remain. The twisted Brauer monoid $\mathcal{B}^\tau_n$ is then nothing but the submonoid of all diagrams having only positive circles or no circles at all. Thus, if elements of $\mathcal{B}^{\pm\tau}_n$ are presented as pairs $(\pi;s)$ with $\pi\in\mathcal{B}_n$ and $s\in\mathbb{Z}$, then the multiplication formula \eqref{eq:mult} persists. Also, the `forgetting' homomorphism $(\pi;s)\mapsto\pi$ of $\mathcal{B}^\tau_n$ onto the Brauer monoid $\mathcal{B}_n$ extends to the monoid $\mathcal{B}^{\pm\tau}_n$.

A further interpretation of the $\pm$-twisted Brauer monoid comes from the considerations in Section~\ref{subsec:background}. Recall that the twisted Brauer monoid $\mathcal{B}^\tau_n$ has been identified with the $F$-basis of the twisted semigroup algebra $(F[X])^\tau[\mathcal{B}_n]$ where the twisting $\tau$ from $\mathcal{B}_n$ to the polynomial ring $F[X]$ is defined by \eqref{eq:xtwist}. If one substitutes $F[X]$ by the ring $F[X,X^{-1}]$ of Laurent polynomials and uses the same twisting $\tau$, the $F$-basis of the twisted semigroup algebra $(F[X,X^{-1}])^\tau[\mathcal{B}_n]$ can be identified with the monoid $\mathcal{B}^{\pm\tau}_n$.

Finally, in terms of `classical' semigroup theory, $\mathcal{B}^{\pm\tau}_n$ is the semigroup of quotients of $\mathcal{B}^\tau_n$ in the sense of Murata~\cite{Murata:50}. In ring theory, it is known that ring identities are preserved by passing to \emph{central localizations}, that is,  rings of quotients over central subsemigroups, see, e.g., \cite[Theorem 3.1]{Rowen:74}. A similar general result holds for semigroups, but we state only a special case that is sufficient for our purposes.

Recall that an identity $w\bumpeq w'$ is \emph{balanced} if for every letter in $\alf(ww')$, the number of its occurrences in $w$ is equal to the number of its occurrences in $w'$.

\begin{lemma}[\!{\mdseries\cite[Lemma 1]{Volkov23}}]
\label{lem:localization}
Suppose that a monoid $\mathcal{S}$ has a submonoid $\mathcal{T}$ such that $\mathcal{S}$ is generated by $\mathcal{T}\cup\{d\}$ for some element $d$ that commutes with every element of $\mathcal{T}$. Then all balanced identities satisfied by $\mathcal{T}$ hold in $\mathcal{S}$ as well.
\end{lemma}

\begin{corollary}
\label{cor:equivalence}
The monoids $\mathcal{B}^{\pm\tau}_n$ and $\mathcal{B}^\tau_n$ satisfy the same identities.
\end{corollary}

\begin{proof}
As observed above, the monoid $\mathcal{B}^{\pm\tau}_n$ has a submonoid isomorphic to $\mathcal{B}^\tau_n$, and the element $d$ commutes with every element of this submonoid and generates $\mathcal{B}^{\pm\tau}_n$ together with this submonoid. By Lemma~\ref{lem:localization}, $\mathcal{B}^{\pm\tau}_n$ satisfies all balanced identities that hold in $\mathcal{B}^\tau_n$. Obviously, the set $\{c^r\mid r\in\mathbb{Z}_{\ge0}\}$ forms a submonoid in $\mathcal{B}^\tau_n$ and this submonoid is isomorphic to the additive monoid of non-negative integers. It is well known that every identity satisfied by the latter monoid is balanced. Hence, so is every identity that holds in $\mathcal{B}^\tau_n$, and we conclude that $\mathcal{B}^{\pm\tau}_n$ satisfies all identities of the monoid $\mathcal{B}^\tau_n$.

The converse statement is obvious as identities are inherited by submonoids.
\end{proof}

\subsection{Structure properties of the monoid $\mathcal{B}^{\pm\tau}_n$}

We have already mentioned that the $\pm$-twisted Brauer monoid $\mathcal{B}^{\pm\tau}_n$ has a `prettier' structure in comparison with $\mathcal{B}^\tau_n$. Here we present some of semigroup-theoretic features of $\mathcal{B}^{\pm\tau}_n$, restricting ourselves to those that are employed in the proof of our main result.

Recall that an element $a$ of a semigroup $\mS$ is said to be \emph{regular} if there exists an element $b\in\mS$ satisfying $aba=a$.  A semigroup is called \emph{regular} if every its element is regular. As the name suggests, regularity is sort of `positive' property. Our first structure observation about $\pm$-twisted Brauer monoids is that they are regular (unlike twisted Brauer monoids that are known to miss this property). We employ the map $\pi\mapsto\pi^*$ on the Brauer monoid $\mathcal{B}_n$ defined as follows. Consider the permutation $^*$ on $[n]\cup [n]'$ that swaps primed with unprimed elements, that is, set $k^*=k'$, $(k')^*=k$ for all $k\in [n]$.
Then define, for $\pi\in\mathcal{B}_n$,
\[
p\mathrel{\pi^*}q\Leftrightarrow {p^*}\mathrel{\pi}{q^*}\ \text{ for all }\ p,q\in[n]\cup [n]'.
\]
Thus, $\pi^*$ is obtained from $\pi$ by interchanging the primed with the unprimed elements. In the geometrical representation of partitions in $\mathcal{B}_n$ via diagrams, the application of ${}^*$ can be visualized as the reflection along the axis between $[n]$ and $[n]'$.

\begin{proposition}
\label{prop:regular}
The monoid $\mathcal{B}^{\pm\tau}_n$ is regular.
\end{proposition}

\begin{proof}
Take an arbitrary element $\xi=(\pi;s)\in\mathcal{B}^{\pm\tau}_n$. Denote by $k$ the number of $t$-wires of the partition $\pi$. Then $n-k$ is even and the number of $\ell$-wires [$r$-wires] in $\pi$ is $m=\frac{n-k}2$. If \wire{p}{q'} is a $t$-wire in $\pi$, then \wire{q}{p'} is a $t$-wire in $\pi^*$ whence \wire{p}{p'} is a $t$-wire in $\pi\pi^*$ and \wire{p}{q'} is a $t$-wire in $\pi\pi^*\pi$. This implies that $\pi$ and $\pi\pi^*\pi$ have the same $t$-wires. Hence the number of $\ell$-wires [$r$-wires] in $\pi\pi^*\pi$ is $m$. Since all $\ell$-wires and $r$-wires of $\pi$ are inherited by $\pi\pi^*\pi$, we conclude that $\pi$ and $\pi\pi^*\pi$ have the same $\ell$-wires and the same $r$-wires. Therefore, $\pi\pi^*\pi=\pi$ in $\mathcal{B}_n$.

If \wire{u'}{v'} is an $r$-wire in $\pi$, then \wire{u}{v} is an $\ell$-wire in $\pi^*$, and gluing these two wires together yields a circle. We see that multiplying $\pi$ by $\pi^*$ on the right produces $m$ circles, that is, $\langle\pi,\pi^*\rangle=m$. As observed in the previous paragraph, $\pi\pi^*$ has $t$-wires of the form \wire{p}{p'} where \wire{p}{q'} is a $t$-wire in $\pi$, whence the number of $t$-wires of $\pi\pi^*$ is at least $k$, and therefore, the number of $r$-wires of $\pi\pi^*$ is at most $m$. Since all $r$-wires of $\pi^*$ are inherited by $\pi\pi^*$, we conclude that the latter partition has no other $r$-wires. If \wire{x'}{y'} is an $r$-wire in $\pi^*$ (and hence in $\pi\pi^*$), then \wire{x}{y} is an $\ell$-wire in $\pi$. Gluing these two wires together yields a circle, whence multiplying $\pi\pi^*$ by $\pi$ on the right produces $m$ circles. Thus, $\langle\pi\pi^*,\pi\rangle=m$. Now let $\eta=(\pi^*;-2m-s)$. Then
\begin{align*}
\xi\eta\xi=(\pi;s)(\pi^*;-2m-s)(\pi;s)&=(\pi\pi^*;-2m+\langle\pi,\pi^*\rangle)(\pi;s)&&\text{by \eqref{eq:mult}} \\
                               &=(\pi\pi^*;-m)(\pi;s) &&\text{since $\langle\pi,\pi^*\rangle=m$}\\
                               &=(\pi\pi^*\pi;-m+s+\langle\pi\pi^*,\pi\rangle)&&\text{by \eqref{eq:mult}} \\
                               &=(\pi\pi^*\pi;s)&&\text{since $\langle\pi\pi^*,\pi\rangle=m$}\\
                               &=(\pi;s)=\xi &&\text{since $\pi\pi^*\pi=\pi$.}
\end{align*}
We see that the element $\xi$ is regular.
\end{proof}

\begin{remark}
It is known and easy to see that the map $\pi\mapsto\pi^*$ is an \emph{involution} of $\mathcal{B}_n$, that is,
\[
\pi^{**}=\pi\ \text{ and }\ (\pi_1\pi_2)^*=\pi_2^*\pi_1^*\  \text{ for all }\ \pi,\pi_1,\pi_2\in\mathcal{B}_n.
\]
In the proof of Proposition~\ref{prop:regular}, we have verified that $\pi\pi^*\pi=\pi$  for all $\pi\in\mathcal{B}_n$. (This fact is also known, but we have included its proof as the argument helps us to calculate the numbers $\langle\pi,\pi^*\rangle$ and $\langle\pi\pi^*,\pi\rangle$.) The three properties of the map $\pi\mapsto\pi^*$ mean that the Brauer monoid $\mathcal{B}_n$ is a \emph{regular $*$-semigroup} as defined in~\cite{nordal_scheiblich}. This stronger form of regularity does not extend to the $\pm$-twisted Brauer monoid $\mathcal{B}^{\pm\tau}_n$. If, as the proof of Proposition~\ref{prop:regular} suggests, one defines the map $\xi=(\pi;s)\mapsto\xi^*:=(\pi^*;t(\pi)-n-s)$ where $t(\pi)$ is the number of $t$-wires of the partition $\pi$, then the equalities $\xi\xi^*\xi=\xi$, $\xi^*\xi\xi^*=\xi^*$, and $(\xi^*)^*=\xi$ hold, but the equality $(\xi_1\xi_2)^*=\xi_2^*\xi_1^*$ fails in general. In fact, if $\xi_1=(\pi_1;s_1)$ and $\xi_2=(\pi_2;s_2)$, then the necessary and sufficient condition for $(\xi_1\xi_2)^*=\xi_2^*\xi_1^*$ to hold is that the number of $r$-wires of $\pi_1$ equals the number of $\ell$-wires of $\pi_2$ and and every $r$-wire of $\pi_1$ is merged with exactly one $\ell$-wire of $\pi_2$ when the product $\pi_1\pi_2$ is formed.
\end{remark}

We proceed with determining the Green structure of the $\pm$-twisted Brauer monoid. Recall the necessary definitions.

Let $\mathcal{S}$ be a semigroup. As usual, $\mathcal{S}^1$ stands for the least monoid containing $\mathcal{S}$ (that is, $\mathcal{S}^1=\mathcal{S}$ if $\mathcal{S}$ is a monoid and otherwise $\mathcal{S}^1=\mathcal{S}\cup\{1\}$ where the new symbol $1$ behaves as a multiplicative identity element). Define three natural preorders $\le_\mathscr{L}$, $\le_\mathscr{R}$ and $\le_\mathscr{J}$ which are the relations of left, right and bilateral divisibility respectively:
\begin{align*}
a\le_\mathscr{L} b &\Leftrightarrow a=sb\ \text{ for some }\ s\in\mathcal{S}^1;\\
a\le_\mathscr{R} b &\Leftrightarrow a=bs\ \text{ for some }\ s\in\mathcal{S}^1;\\
a\le_\mathscr{J} b &\Leftrightarrow a=sbt\ \text{ for some }\
s,t\in\mathcal{S}^1.
\end{align*}
Green's equivalences $\mathscr{L}$, $\mathscr{R}$, and $\mathscr{J}$ are the equivalence relations corresponding to the preorders $\le_\mathscr{L}$, $\le_\mathscr{R}$ and $\le_\mathscr{J}$ (that is, $a\,\mathscr{L}\,b$ if and only if $a\le_\mathscr{L} b\le_\mathscr{L} a$ etc). In addition, let $\mathscr{H}=\mathscr{L}\cap\mathscr{R}$ and $\mathscr{D}=\mathscr{L}\mathscr{R}=\{(a,b)\in\mS\times\mS\mid \exists\,c\in\mS: (a,c)\in\mathscr{L}\land (c,b)\in\mathscr{R}\}$.

Our description of Green's relations on $\mathcal{B}^{\pm\tau}_n$ has the same form as (and easily follows from) the description of Green's relations on the Brauer monoid $\mathcal{B}_n$ in \cite[Section~7]{Ma98}. At the same time, it essentially differs from the description of Green's relations on the twisted Brauer monoid that can be found in~\cite[Section~3.1]{DE18} or \cite[Section 4]{FL11}.

For a partition $\pi\in\mathcal{B}_n$, let $L(\pi)$ and $R(\pi)$ denote the sets of its $\ell$- and, respectively, $r$-wires, and let $t(\pi)$ denote the number of its $t$-wires.

\begin{proposition}
\label{prop:green}
Elements $(\pi_1;s_1),(\pi_2;s_2)\in\mathcal{B}^{\pm\tau}_n$ are
\begin{itemize}
\item[(L)] $\mathscr{L}$-related if and only if $R(\pi_1)=R(\pi_2)$;
\item[(R)] $\mathscr{R}$-related if and only if $L(\pi_1)=L(\pi_2)$;
\item[(H)] $\mathscr{H}$-related if and only if $L(\pi_1)=L(\pi_2)$ and $R(\pi_1)=R(\pi_2)$;
\item[(J)] $\mathscr{J}$-related if and only if they are $\mathscr{D}$-related if and only if $t(\pi_1)=t(\pi_2)$.
\end{itemize}
\end{proposition}

\begin{proof}
(L) If $(\pi_1;s_1)$ and $(\pi_2;s_2)$ are $\mathscr{L}$-related in $\mathcal{B}^{\pm\tau}_n$, then their images $\pi_1$ and $\pi_2$ under the `forgetting' homomorphism $(\pi;s)\mapsto\pi$ are $\mathscr{L}$-related in $\mathcal{B}_n$. By \cite[Theorem~7(1)]{Ma98}, this implies $R(\pi_1)=R(\pi_2)$.

Conversely, suppose that $R(\pi_1)=R(\pi_2)$. Then $\pi_1$ and $\pi_2$ are $\mathscr{L}$-related in $\mathcal{B}_n$ by \cite[Theorem~7(1)]{Ma98}. This means that $\sigma_1\pi_1=\pi_2$ and $\sigma_2\pi_2=\pi_1$ for some partitions $\sigma_1,\sigma_2\in\mathcal{B}_n$. Let $r_1=\langle\sigma_1,\pi_1\rangle$ and $r_2=\langle\sigma_2,\pi_2\rangle$. Then, using \eqref{eq:mult}, we get
\begin{gather*}
(\sigma_1;s_2-s_1-r_1)(\pi_1;s_1)=(\sigma_1\pi_1;s_2-r_1+\langle\sigma_1,\pi_1\rangle)=(\pi_2;s_2),\\
(\sigma_2;s_1-s_2-r_2)(\pi_2;s_2)=(\sigma_2\pi_2;s_1-r_2+\langle\sigma_2,\pi_2\rangle)=(\pi_1;s_1).
\end{gather*}
Hence, $(\pi_1;s_1)$ and $(\pi_2;s_2)$ are $\mathscr{L}$-related in $\mathcal{B}^{\pm\tau}_n$.

(R) follows by a symmetric argument.

(H) is clear.

(J) The `only if' part follows from the fact that $t(\pi\sigma)\le\min\{t(\pi),t(\sigma)\}$ for all partitions $\pi,\sigma\in\mathcal{B}_n$.

For the `if' part, suppose that $t(\pi_1)=t(\pi_2)=k$. Then $n-k$ is even and the number of $\ell$-wires [$r$-wires] in $\pi_1$ and $\pi_2$ is $m=\frac{n-k}2$. Construct a partition $\sigma$ as follows. Take the same $r$-wires as in $\pi_1$ and the same $\ell$-wires as in $\pi_2$. After that, there remain $k=n-2m$ 'non-engaged' left points and the same number of 'non-engaged' right points; to complete the construction, we couple them into $t$-wires in any of $k!$ possible ways. Then $(\pi_1;s)$ and $(\sigma;0)$ are $\mathscr{L}$-related in $\mathcal{B}^{\pm\tau}_n$ by (L), while $(\sigma,0)$ and $(\pi_2;s_2)$ are $\mathscr{R}$-related in $\mathcal{B}^{\pm\tau}_n$ by (R). Therefore, $(\pi_1;s_1)$ and  $(\pi_2;s_2)$ are $\mathscr{D}$-related in $\mathcal{B}^{\pm\tau}_n$.
\end{proof}

Comparing Proposition~\ref{prop:green} and the description of Green's relations on the Brauer monoid in \cite[Section~7]{Ma98}, we state the following.
\begin{corollary}
\label{cor:the same}
For any Green relation $\K\in\{\gJ,\gD,\gL,\gR,\gH\}$, one has $(\pi_1;s_1)\,\mathrsfs{K}\,(\pi_2;s_2)$ in $\mathcal{B}^{\pm\tau}_n$ if and only if $\pi_1\,\mathrsfs{K}\,\pi_2$ in $\mathcal{B}_n$.
\end{corollary}

By Proposition~\ref{prop:green}(J), the monoid $\mathcal{B}^{\pm\tau}_n$ has the following $\mathscr{J}$-classes:
\[
J_k=\{(\pi;s)\in\mathcal{B}_n\times\mathbb{Z} \mid t(\pi)=k\}.
\]
Here $k=n,n-2,\dots,0$ if $n$ is even and $k=n,n-2,\dots,1$ if $n$ is odd so the number of $\mathscr{J}$-classes is $\lceil\frac{n+1}2\rceil$. For comparison, the twisted Brauer monoid $\mathcal{B}^\tau_n$ contains infinite descending chains of $\mathscr{J}$-classes; see \cite[Section 3.1]{DE18}.

A semigroup $\mS$ is called \emph{stable} if for all $a,b\in\mS$, the implications $a\,\mathscr{J}\,ab\Rightarrow a\,\mathscr{R}\,ab$ and $a\,\mathscr{J}\,ba\Rightarrow a\,\mathscr{L}\,ba$ hold. It is well-known that every finite semigroup is stable, see, e.g., \cite{EH20}.

\begin{corollary}
\label{cor:stable}
The monoid $\mathcal{B}^{\pm\tau}_n$ is stable.
\end{corollary}

\begin{proof}
Let $a = (\pi_1;s_1)$ and $b = (\pi_2;s_2)$ be such that $a\,\mathscr{J}\,ab$ in $\mathcal{B}^{\pm\tau}_n$. Then $\pi_1\,\mathscr{J}\,\pi_1\pi_2$ in $\mathcal{B}_n$ by Corollary~\ref{cor:the same} whence $\pi_1\,\mathscr{R}\,\pi_1\pi_2$ in $\mathcal{B}_n$ since the finite monoid $\mathcal{B}_n$ is stable. Now Corollary~\ref{cor:the same} gives $a\,\mathscr{R}\,ab$ in  $\mathcal{B}^{\pm\tau}_n$. The other implication from the definition of stability is verified in the same way.
\end{proof}

The final ingredient that we need is the structure of the maximal subgroups of $\mathcal{B}^{\pm\tau}_n$.

\begin{proposition}
\label{prop:subgroups}
For $k>0$, the maximal subgroups in the $\mathscr{J}$-class $J_k$ of the monoid $\mathcal{B}^{\pm\tau}_n$ are isomorphic to the group $\mathbb{S}_k\times\mathbb{Z}$.
\end{proposition}

\begin{proof}
By Green's Theorem \cite[Theorem 2.2.5]{Howie:1995}, the maximal subgroups in $J_k$ are exactly the $\mathscr{H}$-classes in $J_k$ that contain idempotents. If $\epsilon=(\pi;s)$ is an idempotent in $J_k$, then $(\pi;s)=(\pi;s)^2=(\pi^2,2s+\langle\pi,\pi\rangle$ by~\eqref{eq:mult}. Hence $\pi=\pi^2$ is an idempotent in the monoid $\mathcal{B}_n$, and $s=2s+\langle\pi,\pi\rangle$, whence $\langle\pi,\pi\rangle=-s$. By Corollary~\ref{cor:the same}, the $\mathscr{H}$-class of the idempotent $\epsilon$ is the set $\mH_\epsilon=\{(\sigma; i)\mid\sigma\in\mH_\pi;\, i \in \mathbb{Z}\}$, where $\mH_\pi$ is the $\mathscr{H}$-class of the idempotent $\pi$ in $\mathcal{B}_n$. By Proposition~\ref{prop:green}(H) all partitions  $\sigma_1,\sigma_2\in\mH_\pi$ have the same $\ell$- and $r$-wires as $\pi$. Hence $\langle\sigma_1,\sigma_2\rangle=\langle\pi,\pi\rangle=-s$, and therefore,
\[
(\sigma_1; i_1)(\sigma_2; i_2) = (\sigma_1\sigma_2; i_1 + i_2 + \langle\sigma_1,\sigma_2\rangle)=(\sigma_1\sigma_2; i_1 + i_2 - s).  
\]
This readily implies that the bijection $\mH_\epsilon\to \mH_\pi\times\mathbb{Z}$ defined by $(\sigma; i)\mapsto (\sigma; i-s)$ is a group isomorphism. Since the number of $t$-wires in each partition in $\mH_\pi$ is equal to $k$, \cite[Theorem 1]{Ma98} implies that the subgroup $\mH_\pi$ is isomorphic to the symmetric group $\mathbb{S}_k$.
\end{proof}

\begin{remark}
\label{rem:S0}
If $n$ is even, the least $\mathscr{J}$-class of the monoid $\mathcal{B}^{\pm\tau}_n$ (with respect to the ordering of $\mathscr{J}$-classes induced by the preorder $\le_{\mathscr{J}}$) is $J_0$. Its maximal subgroups are isomorphic to $\mathbb{Z}$. For the sake of uniformity, let $\mathbb{S}_0$ be the trivial group (this complies with the usual convention that $0!=1$). This way Proposition~\ref{prop:subgroups} extends to the case $k=0$.
\end{remark}

\section{Reduction theorem for identity checking}
\label{sec:almeida}

We need the following reduction:
\begin{theorem}
\label{thm:reduction}
Let $\mathcal{S}$ be a stable semigroup with finitely many $\mathscr{J}$-classes and $\mathcal{G}$ the direct product of all maximal subgroups of $\mS$. Then there exists a polynomial reduction from the problem \textsc{Check-Id}$(\mathcal{G})$ to the problem \textsc{Check-Id}$(\mathcal{S})$.
\end{theorem}

In \cite[Theorem 1]{AVG09}, the same reduction was proved for \textbf{finite} \sgps. In fact, the proof in \cite{AVG09} needs only minor adjustments to work under the premises of Theorem~\ref{thm:reduction}. Still, for the reader's convenience, we provide a self-contained argument so that it should be possible to understand the proof of Theorem~\ref{thm:reduction} without any acquaintance with~\cite{AVG09}.

\begin{proof}
The existence of a polynomial reduction from \textsc{Check-Id}$(\mathcal{G})$ to \textsc{Check-Id}$(\mathcal{S})$ means the following. Given an arbitrary instance of \textsc{Check-Id}$(\mathcal{G})$, i.e., an arbitrary identity $u\bumpeq v$, one can construct an identity $U\bumpeq V$ such that:
\begin{itemize}
  \item[(Size)] the lengths of the words $U$ and $V$ are bounded by the values of a fixed polynomial in the lengths of the words $u$ and $v$;
  \item[(Equi)] the identity $U\bumpeq V$ holds in $\mS$ if and only if the identity $u\bumpeq v$ holds in $\mG$.
\end{itemize}

Towards the construction, assume that $\Sigma=\alf(uv)$ consists of the letters $x_1,\dots,x_m$. Let $\Sigma^+$ denote the \emph{free semigroup over} $\Sigma$, that is, the set of all words built from the letters in $\Sigma$ and equipped with concatenation as multiplication. It is known (and easy to verify) that $\Sigma^+$ has the following \emph{universal property}: every map $\Sigma\to \Sigma^+$ uniquely extends to an endomorphism of the semigroup $\Sigma^+$.

Define the following $m$ words:
\begin{align}
\label{eq:words}
w_1&=x_1^2x_2\cdots x_mx_1,\notag\\
w_2&=x_1x_2^2\cdots x_mx_1,\notag\\
\ldots&\hbox to 2.5cm{\dotfill}\\
w_{m-1}&=x_1x_2\cdots x_{m-1}^2x_mx_1,\notag\\
w_m&=x_1x_2\cdots x_mx_1.\notag
\end{align}
(The reader might suspect a typo in the last line of \eqref{eq:words} as the word $w_m$ involves no squared letter, unlike all previous words. No, the expression for $w_m$ is correct, and its distinct role will be revealed shortly.) We denote by $\varphi$ the endomorphism of $\Sigma^+$ that extends the map $x_i\mapsto w_i$, $i=1,\dots,m$. For each $k=1,2,\dotsc$, let $w_{i,k}=\varphi^k(x_i)$ and let $N$ be the number of $\mathscr{J}$-classes of $\mS$. We claim that the identity
\[
U=u(w_{1,2N},\dots,w_{m,2N})\bumpeq v(w_{1,2N},\dots,w_{m,2N})=V
\]
possesses the desired properties (Size) and (Equi).

For (Size), observe that the length of each of the words~\eqref{eq:words} does not exceed $m+2$, and therefore, the length of each of the words $w_{1,2N},\dots,w_{m,2N}$ does not exceed $(m+2)^{2N}$. Here, the number $N$ is defined by the semigroup $\mathcal{S}$ only and does not depend on the words $u$ and $v$, and the number $m$ does not exceed the maximum of the lengths of $u$ and $v$. Since the length of the word $U=u(w_{1,2N},\dots,w_{m,2N})$ (respectively, $V=v(w_{1,2N},\dots,w_{m,2N})$) does not exceed the product of the maximum length of the words $w_{i,2N}$ and the length of the word $u$ (respectively, $v$), the polynomial $X^{2N+1}$ witnesses the property (Size).

The verification of (Equi) is more involved. We start with the following observation.
\begin{lemma}
\label{lem:almeida}
If $\mathcal{S}$ a stable semigroup with a finite number $N$ of $\mathscr{J}$-classes, then for every substitution $\Sigma\to\mathcal{S}$, there is a subgroup $\mathcal{H}$ in $\mathcal{S}$ such that the values of all words $w_{1,2N},\dots,w_{m,2N}$ under this substitution belong to $\mathcal{H}$.
\end{lemma}

\begin{proof}
Notice that for each $k=1,2,\dotsc$,
\begin{multline}
\label{eq:endo}
w_{i,k+1}=\varphi^{k+1}(x_i)=\varphi^k(\varphi(x_i))=\varphi^k(w_i(x_1,\dots,x_m))\\
{}=w_i(\varphi^k(x_1),\dots,\varphi^k(x_m))=w_i(w_{1,k},\dots,w_{m,k}).
\end{multline}
Inspecting the definition \eqref{eq:words}, we see that every letter $x_i$ occurs in each of the words $w_1,\dots,w_m$. Therefore, the equalities~\eqref{eq:endo} imply that the word $w_{i,k}$ appears as a factor in the word $w_{j,k+1}$ for every $k=1,2,\dots$ and every $i,j=1,\dots,m$.

Fix a substitution $\Sigma\to\mathcal{S}$ and denote the value of a word $w\in\Sigma^+$ under this substitution by $\overline{w}$. Since $w_{1,k}$ appears as a factor in $w_{1,k+1}$, the following inequalities hold in  $\mathcal{S}$:
\[
\overline{w_{1,1}}\ge_\mathscr{J}\overline{w_{1,2}}\ge_\mathscr{J}\cdots\ge_\mathscr{J}\overline{w_{1,2N+1}}.
\]
Amongst these inequalities, at most $N-1$ can be strict, whence by the pigeonhole principle, the sequence $\overline{w_{1,1}},\overline{w_{1,2}},\dots,\overline{w_{1,2N+1}}$ contains three adjacent $\mathscr{J}$-related elements. Let $k<2N$ be such that $\overline{w_{1,k}}\,\mathscr{J}\,\overline{w_{1,k+1}}\,\mathscr{J}\,\overline{w_{1,k+2}}$. Again inspecting \eqref{eq:words}, we see that the word $x_1^2$ appears as a factor in the word $w_1$. Hence, by the equalities~\eqref{eq:endo}, the word $w_{1,k}^2$ appears as a factor in the word $w_{1,k+1}$, and therefore, we have $\overline{w_{1,k}}^2\ge_\mathscr{J}\overline{w_{1,k+1}}\,\mathscr{J}\,\overline{w_{1,k}}$ in $\mS$. Obviously, $\overline{w_{1,k}}^2\le_\mathscr{J}\overline{w_{1,k}}$ whence $\overline{w_{1,k}}^2\,\mathscr{J}\,\overline{w_{1,k}}$. Since $\mathcal{S}$ is stable, $\overline{w_{1,k}}^2\,\mathscr{J}\,\overline{w_{1,k}}$ implies $\overline{w_{1,k}}^2\,\mathscr{L}\,\overline{w_{1,k}}$ and $\overline{w_{1,k}}^2\,\mathscr{R}\,\overline{w_{1,k}}$, that is, $\overline{w_{1,k}}^2\,\mathscr{H}\,\overline{w_{1,k}}$. By Green's Theorem \cite[Theorem 2.2.5]{Howie:1995}, the $\mathscr{H}$-class $\mathcal{H}$ of the element $\overline{w_{1,k}}$ is a maximal subgroup of the semigroup $\mathcal{S}$.

Yet another look at \eqref{eq:words} reveals that each of the words $w_1,\dots,w_m$ starts and ends with the letter $x_1$. In view of the equalities~\eqref{eq:endo}, the word $w_{1,k}$ appears as a prefix as well as a suffix of each of the words $w_{i,k+1}$, which, in turn, appear as factors in the word $w_{1,k+2}$. Hence $\overline{w_{i,k+1}}=\overline{w_{1,k}}b=a\overline{w_{1,k}}$ for some $a,b\in\mS$ and all elements $\overline{w_{i,k+1}}$ lie in the $\mathscr{J}$-class of $\overline{w_{1,k}}$. By stability of the semigroup $\mS$, all these elements lie in both the $\mathscr{L}$-class and the $\mathscr{R}$-class of the element $\overline{w_{1,k}}$. Thus, all elements $\overline{w_{i,k+1}}$ lie in the subgroup $\mathcal{H}$, whence the subgroup contains all elements $\overline{w_{i,\ell}}$ for all $\ell>k$. We see that the subgroup $\mathcal{H}$ indeed contains the values of all words $w_{1,2N},\dots,w_{m,2N}$ under the substitution we consider.
\end{proof}

Now we are in a position to prove that if the identity $u\bumpeq v$ holds in $\mathcal{G}$, then the identity $U\bumpeq V$ holds in $\mathcal{S}$. Consider an arbitrary substitution $\zeta\colon\Sigma\to\mathcal{S}$. By Lemma~\ref{lem:almeida}, the values of the words $w_{1,2N},\dots,w_{m,2N}$ under $\zeta$ lie in a subgroup $\mathcal{H}$ of the semigroup $\mathcal{S}$. Since $\mathcal{H}$ is a subgroup of $\mathcal{G}$, the identity $u\bumpeq v$ holds in $\mathcal{H}$, and hence, substituting for $x_1,\dots,x_m$ the values of the words $w_{1,2N},\dots,w_{m,2N}$ yield the equality
\[
u(\zeta(w_{1,2N}),\dots,\zeta(w_{m,2N}))=v(\zeta(w_{1,2N}),\dots,\zeta(w_{m,2N})).
\]
in $\mathcal{H}$. However,
\begin{gather*}
u(\zeta(w_{1,2N}),\dots,\zeta(w_{m,2N}))=\zeta(u(w_{1,2N},\dots,w_{m,2N}))=\zeta(U),\\
v(\zeta(w_{1,2N}),\dots,\zeta(w_{m,2N}))=\zeta(v(w_{1,2N},\dots,w_{m,2N}))=\zeta(V),
\end{gather*}
and hence $U$ and $V$ take the same value under $\zeta$. Since the substitution was arbitrary, the identity $U\bumpeq V$ holds in $\mathcal{S}$.

It remains to verify the converse: if the identity $U\bumpeq V$ holds in $\mathcal{S}$, then the identity $u\bumpeq v$ holds in $\mathcal{G}$. As identities are inherited by direct products, it suffices to show that $u\bumpeq v$  holds in every maximal subgroup $\mathcal{H}$ of $\mathcal{S}$. This amounts to verifying that $u(h_1,\dots,h_m)=v(h_1,\dots,h_m)$ for an arbitrary $m$-tuple of elements $h_1,\dots,h_m\in\mathcal{H}$.

The free semigroup $\Sigma^+$ can be considered as a subsemigroup in the free group $\mathcal{FG}(\Sigma)$ over $\Sigma$. The endomorphism $\varphi\colon x_i\mapsto w_i$ of $\Sigma^+$ extends to an endomorphism of $\mathcal{FG}(\Sigma)$, still denoted by $\varphi$. The words $w_1,\dots,w_m$ defined by \eqref{eq:words} generate $\mathcal{FG}(\Sigma)$ since in $\mathcal{FG}(\Sigma)$, one can express $x_1,\dots,x_m$ via $w_1,\dots,w_m$ as follows:
\begin{align*}
x_1&=w_1w_m^{-1},\\
x_2&=x_1^{-1}w_2w_m^{-1}x_1,\\
x_3&=(x_1x_2)^{-1}w_3w_m^{-1}x_1x_2,\\
\ldots&\hbox to 3.5cm{\dotfill}\\
x_{m-1}&=(x_1x_2\cdots x_{m-2})^{-1}w_{m-1}w_m^{-1}x_1x_2\cdots x_{m-2},\\
x_m&=(x_1x_2\cdots x_{m-1})^{-1}w_mx_1^{-1}.
\end{align*}
(This is where the distinct expression for $w_m$ comes into play!) Hence $\varphi$ treated as an endomorphism of $\mathcal{FG}(\Sigma)$ is surjective, and so is any power of $\varphi$. It is well known (cf.\ \cite[Proposition~I.3.5]{LS80}) that every surjective endomorphism of a finitely generated free group is an automorphism. Denote by $\varphi^{-2N}$ the inverse of the automorphism $\varphi^{2N}$ of $\mathcal{FG}(\Sigma)$ and let $g_i=\varphi^{-2N}(x_i)$, $i=1,\dots,m$. Then
\begin{multline}
\label{eq:auto}
w_{i,2N}(g_1,\dots,g_m)=w_{i,2N}(\varphi^{-2N}(x_1),\dots,\varphi^{-2N}(x_m))\\
{}=\varphi^{-2N}(w_{i,2N}(x_1,\dots,x_m))=\varphi^{-2N}(\varphi^{2N}(x_i))=x_i
\end{multline}
for all $i=1,\dots,m$. Since the equalities~\eqref{eq:auto} hold in the free $m$-generated group, they remain valid under any interpretation of the letters $x_1,\dots,x_m$ by arbitrary $m$ elements of an arbitrary group. Now we define a substitution $\zeta\colon\Sigma\to\mathcal{H}$ letting
\[
\zeta(x_i)=g_i(h_1,\dots,h_m),\ \ i=1,\dots,m.
\]
Then in view of~\eqref{eq:auto} we have
\begin{multline*}
\zeta(w_{i,2N}(x_1,\dots,x_m))=w_{i,2N}(\zeta(x_1),\dots,\zeta(x_m))\\
{}=w_{i,2N}\bigl(g_1(h_1,\dots,h_m),\dots,g_m(h_1,\dots,h_m)\bigr)=h_i
\end{multline*}
for all $i=1,\dots,m$. Hence we have
\begin{align*}
u(h_1,\dots,h_m)&=u(\zeta(w_{1,2N}(x_1,\dots,x_m)),\dots,\zeta(w_{m,2N}(x_1,\dots,x_m)))\\
&=\zeta(u(w_{1,2N}(x_1,\dots,x_m),\dots,w_{m,2N}(x_1,\dots,x_m)))\\
&=\zeta(U(x_1,\dots,x_m)),
\end{align*}
and, similarly, $v(h_1,\dots,h_m)=\zeta(V(x_1,\dots,x_m))$. Since the identity $U\bumpeq V$ holds in $\mathcal{S}$, the values of the words $U$ and $V$ under $\zeta$ are equal, whence $u(h_1,\dots,h_m)=v(h_1,\dots,h_m)$, as required. This completes the proof of (Equi), and hence, the proof of Theorem~\ref{thm:reduction}.
\end{proof}

\section{Co-NP-completeness of identity checking in $\mathcal{B}^\tau_n$ with $n\ge 5$}
\label{sec:main}

We are ready to prove our main result.

\begin{theorem}
\label{thm:main}
For each $n\ge 5$, the problem \textsc{Check-Id}$(\mathcal{B}^\tau_n)$ is co-NP-complete.
\end{theorem}

\begin{proof}
Proving that a decision problem \textsc{P} is co-NP-complete amounts to showing that the problem \textsc{P} belongs to the complexity class co-NP and is co-NP-hard, the latter meaning that there exists a polynomial reduction from a co-NP-complete problem to \textsc{P}.

The fact that \textsc{Check-Id}$(\mathcal{B}^\tau_n)$ lies in the class co-NP is easy. The following non-deterministic algorithm has a chance to return the answer ``NO'' if and only if it is given an identity $w\bumpeq w'$ that does not hold in $\mathcal{B}^\tau_n$.

1. If $|\alf(ww')|=k$, guess a $k$-tuple of elements in $\mathcal{B}^\tau_n$.

2. Substitute the elements from the guessed $k$-tuple for the letters in $\alf(ww')$ and compute the values of the words $w$ and $w'$.

3. Return ``NO'' if the values are different.

The multiplication in $\mathcal{B}^\tau_n$ is constructive so that the computation in Step 2 takes polynomial (in fact, linear) time in the lengths of $w$ and $w'$.

In order to prove the co-NP-hardness of \textsc{Check-Id}$(\mathcal{B}^\tau_n)$, we use the reduction of Theorem~\ref{thm:reduction} and the powerful result by Horv\'ath, Lawrence, Merai, and Szab\'o~\cite{HLMS07} who discovered that for every nonsolvable finite group $\mathcal{G}$,  the problem \textsc{Check-Id}($\mathcal{G}$) is co-NP-complete. Already Galois knew that for $n\ge 5$ the group $\mathbb{S}_n$ is nonsolvable so the result of~\cite{HLMS07} applies to $\mathbb{S}_n$.

Fix an $n\ge 5$. An identity $w\bumpeq w'$ holds in the group $\mathbb{S}_n$ if and only if so does the identity $w^{n!-1}w'\bumpeq 1$. The length of the word $w^{n!-1}w'$ is bounded by the value of the polynomial $X^{n!}$ in the maximum length of the words $w$ and $w'$. Thus, we have a mutual polynomial reduction between the problem \textsc{Check-Id}($\mathbb{S}_n$) and the problem of determining whether or not all values of a given semigroup word $v$ in $\mathbb{S}_n$ are equal to the identity of the group. It is well known that the~center of $\mathbb{S}_n$ is trivial whence the latter property is equivalent to saying that all values of $v$ in $\mathbb{S}_n$ lie in the center. This, in turn, is equivalent to the fact that the identity $vx\bumpeq xv$ where $x\notin\alf(v)$ holds in $\mathbb{S}_n$. Clearly, for any word $v$, the identity $vx\bumpeq xv$ is balanced. We conclude that the problem \textsc{Check-Id}($\mathbb{S}_n$) remains co-NP-complete when restricted to balanced identities.

An identity holds in the direct product of semigroups if and only if it holds in each factor of the product. Applying this to the product $\mathbb{S}_n\times\mathbb{Z}$ and taking into account that $\mathbb{Z}$ satisfies exactly balanced identities, we see that the identities holding in $\mathbb{S}_n\times\mathbb{Z}$ are precisely the balanced identities holding in $\mathbb{S}_n$. Hence the problem \textsc{Check-Id}($\mathbb{S}_n\times\mathbb{Z}$) is co-NP-complete.

By Proposition~\ref{prop:subgroups} (and Remark~\ref{rem:S0}), the maximal subgroups of the $\pm$-twisted Brauer monoid $\mathcal{B}^{\pm\tau}_n$ are of the form $\mathbb{S}_k\times\mathbb{Z}$, where $k=n,n-2,\dots,0$ if $n$ is even and $k=n,n-2,\dots,1$ if $n$ is odd. Any group of this form embeds into $\mathbb{S}_n\times\mathbb{Z}$ whence the identities that hold in each maximal subgroup of $\mathcal{B}^{\pm\tau}_n$ are exactly the identities of $\mathbb{S}_n\times\mathbb{Z}$. We conclude that the identities of the direct product $\mG$ of all maximal subgroups of $\mathcal{B}^{\pm\tau}_n$ coincide with the identities of $\mathbb{S}_n\times\mathbb{Z}$. Hence, the problem \textsc{Check-Id}($\mG$) is co-NP-complete.

By Corollary~\ref{cor:stable} and Proposition~\ref{prop:green}, the $\pm$-twisted Brauer monoid is stable and has finitely many  $\mathscr{J}$-classes. Thus, Theorem~\ref{thm:reduction} applies to $\mathcal{B}^{\pm\tau}_n$, providing a polynomial reduction from the co-NP-complete problem \textsc{Check-Id}$(\mathcal{G})$ to the problem \textsc{Check-Id}$(\mathcal{B}^{\pm\tau}_n)$. Hence, the latter problem is co-NP-hard. It remains to refer to Corollary~\ref{cor:equivalence} stating that the $\pm$-twisted Brauer monoid $\mathcal{B}^{\pm\tau}_n$ and the twisted Brauer monoid $\mathcal{B}^\tau_n$ satisfy the same identities, and therefore, the problem \textsc{Check-Id}$(\mathcal{B}^\tau_n)$ is co-NP-hard as well.
\end{proof}

The restriction $n\ge5$ in Theorem~\ref{thm:main} is essential for the above proof. This does not mean, however, that it is necessary for co-NP-completeness of the problem \textsc{Check-Id}$(\mathcal{B}^\tau_n)$. In fact, we have proved that identity checking in $\mathcal{B}^\tau_4$ remains co-NP-complete. The proof uses a completely different technique, and therefore, it will be published separately.

The case $n=3$ remains open. As for $n=1,2$, the monoid $\mathcal{B}^\tau_1$ is trivial, and hence, it satisfies every identity, and the monoid $\mathcal{B}^\tau_2$ is commutative and can be easily shown to satisfy exactly balanced identities. Thus, for $n=1,2$, the problem \textsc{Check-Id}$(\mathcal{B}^\tau_n)$ is polynomial (actually, linear) time decidable.

\begin{remark}
\label{new identities}
Up to now, the only available information about the identities of twisted Brauer monoids was \cite[Theorem 4.1]{ACHLV15} showing that no finite set of identities of $\mathcal{B}^\tau_n$ with $n\ge3$ can infer all such identities, in other words, $\mathcal{B}^\tau_n$ with $n\ge3$ has no finite identity basis. This fact was obtained via a `high-level' argument that allows one to prove, under certain conditions, that a semigroup $\mathcal{S}$ admits no finite identity basis, without writing down any concrete identity holding in $\mathcal{S}$. In contrast, the above proof of Theorem~\ref{thm:main} via Theorem~\ref{thm:reduction} is constructive. Following the recipe of Theorem~\ref{thm:reduction}, one can use the words \eqref{eq:words} to convert any concrete balanced semigroup identity $u\bumpeq v$ of the group $\mathbb{S}_n$ into an identity $U\bumpeq V$ of the twisted Brauer monoid $\mathcal{B}^\tau_n$. We refer the reader to~\cite{BKSS17,KS21} for recent information about short balanced semigroup identities in symmetric groups.
\end{remark}

\section{Related results and further work}
\label{sec:conclusion}

\subsection{Checking identities in twisted partition monoids} The approach of the present paper can be applied to studying the identities of other interesting families of infinite monoids, in particular, twisted partition monoids. The latter constitute a natural generalization of twisted Brauer monoids and also serve as bases of certain semigroup algebras relevant in statistical mechanics and representation theory, the so-called \emph{partition algebras}. Partition algebras were discovered and studied in depth by Martin \cite{Martin91,Martin94,Martin96,Martin00} and, independently, by Jones \cite{Jones94} in the context of  statistical mechanics; their remarkable role in representation theory is nicely presented in the introduction of~\cite{HR05}.

We define twisted partition monoids, `twisting' the definition of partition monoids as given in \cite{Wil07}. As in Section~\ref{subsec:twisted}, let $[n]=\{1,\dots,n\}$ and $[n]'=\{1',\dots,n'\}$. Consider the set $\mathcal{P}^\tau_n$ of all pairs $(\pi;s)$ where $\pi$ is an arbitrary partition of the $2n$-element set $[n]\cup [n]'$ and $s$ is a nonnegative integer. (The difference with $\mathcal{B}^\tau_n$ is that one drops the restriction that all blocks of $\pi$ consist of two elements.) The product $(\pi;s)$ of two pairs $(\pi_1;s_1),(\pi_2;s_2)\in\mathcal{P}^\tau_n$ is computed in the following six steps.
\begin{enumerate}
\item Let $[n]''=\{1'',\dots,n''\}$ and define the partition $\pi_2'$ on ${[n]'}\cup {[n]''}$ by
\[
{x'}\mathrel{\pi_2'}{y'}\Leftrightarrow x\mathrel{\pi_2} y\text{ for all } x,y\in[n]\cup [n]'.
\]
\item Let $\pi''$ be the equivalence relation on $[n]\cup [n]'\cup[n]''$ generated by $\pi_1\cup\pi_2'$, that is, $\pi''$ is the transitive closure of $\pi_1\cup\pi_2'$.
\item Count the number of blocks of $\pi''$ that involve only elements from $[n]'$ and denote this number by $\langle\pi_1,\pi_2\rangle$.
\item Convert $\pi''$ into a partition $\pi'$ on the set $[n]\cup[n]''$ by removing all elements having a single prime $'$ from all blocks; all blocks having only such elements are removed as a whole.
\item Replace double primes with single primes to obtain a partition $\pi$, that is, set
\[
x\mathrel{\pi}y\Leftrightarrow f(x)\mathrel{\pi'}f(y) \text{ for all }x,y\in[n]\cup [n]'
\]
where $f\colon[n]\cup [n]'\to [n]\cup{[n]''}$ is the bijection $x\mapsto x$, $x'\mapsto x''$ for all  $x\in [n]$.
\item Set $(\pi_1;s_1)(\pi_2;s_2)=(\pi;s_1+s_2+\langle\pi_1,\pi_2\rangle)$.
\end{enumerate}

For an illustration, let $n=5$ and consider $(\pi_1;s_1),(\pi_2;s_2)\in\mathcal{P}^\tau_5$ with
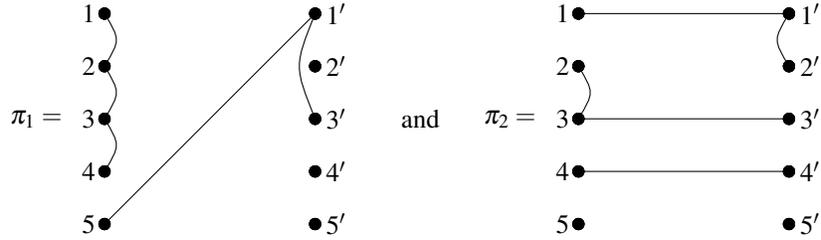
\begin{center}
\begin{tikzpicture}
[scale=0.7]
\node[] at (-1.3, -2) {$\pi_1 = $};
\node[] at (-0.3, 0) {$1$};
\node[] at (-0.3, -1) {$2$};
\node[] at (-0.3, -2) {$3$};
\node[] at (-0.3, -3) {$4$};
\node[] at (-0.3, -4) {$5$};
\node[] at (4.4, 0) {$1^\prime$};
\node[] at (4.4, -1) {$2^\prime$};
\node[] at (4.4, -2) {$3^\prime$};
\node[] at (4.4, -3) {$4^\prime$};
\node[] at (4.4, -4) {$5^\prime$};
\foreach \x in {0, 0, 0, 0} \foreach \y in {0, -1, -2, -3, -4} \filldraw (\x,\y) circle (3pt);
\foreach \x in {4, 4, 4, 4} \foreach \y in {0, -1, -2, -3, -4} \filldraw (\x,\y) circle (3pt);
\draw (0, -4) -- (4, 0);
\draw (4, 0) .. controls (3.6, -1) ..  (4, -2);
\draw (0, 0) .. controls (0.3, -0.5) ..  (0, -1);
\draw (0, -1) .. controls (0.3, -1.5) ..  (0, -2);
\draw (0, -2) .. controls (0.3, -2.5) ..  (0, -3);
\node[] at (6, -2) {and};
\node[] at (7.7, -2) {$\pi_2 = $};
\node[] at (8.7, 0) {$1$};
\node[] at (8.7, -1) {$2$};
\node[] at (8.7, -2) {$3$};
\node[] at (8.7, -3) {$4$};
\node[] at (8.7, -4) {$5$};
\node[] at (13.4, 0) {$1^\prime$};
\node[] at (13.4, -1) {$2^\prime$};
\node[] at (13.4, -2) {$3^\prime$};
\node[] at (13.4, -3) {$4^\prime$};
\node[] at (13.4, -4) {$5^\prime$};
\foreach \x in {9, 9, 9, 9} \foreach \y in {0, -1, -2, -3, -4} \filldraw (\x,\y) circle (3pt);
\foreach \x in {13, 13, 13, 13} \foreach \y in {0, -1, -2, -3, -4} \filldraw (\x,\y) circle (3pt);
\draw (9, 0) -- (13, 0);
\draw (13, 0) .. controls (12.7, -0.5) ..  (13, -1);
\draw (9, -2) -- (13, -2);
\draw (9, -1) .. controls (9.3, -1.5) ..  (9, -2);
\draw (9, -3) -- (13, -3);
\end{tikzpicture}
\end{center}
where the partitions are shown as graphs on $[5]\cup [5]'$ whose connected components represent blocks. Then
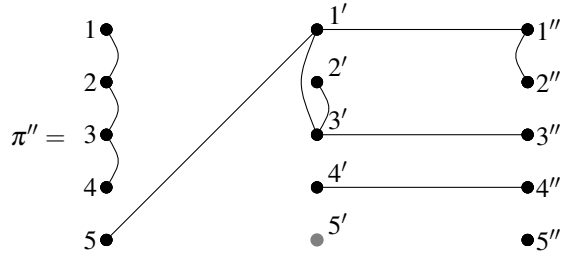
\begin{center}
\begin{tikzpicture}
[scale=0.7]
\node[] at (-1.3, -2) {$\pi^{\prime \prime} = $};
\node[] at (-0.3, 0) {$1$};
\node[] at (-0.3, -1) {$2$};
\node[] at (-0.3, -2) {$3$};
\node[] at (-0.3, -3) {$4$};
\node[] at (-0.3, -4) {$5$};
\node[] at (4.4, 0.3) {$1^\prime$};
\node[] at (4.4, -0.7) {$2^\prime$};
\node[] at (4.4, -1.7) {$3^\prime$};
\node[] at (4.4, -2.7) {$4^\prime$};
\node[] at (4.4, -3.7) {$5^\prime$};
\node[] at (8.4, 0) {$1^{\prime \prime}$};
\node[] at (8.4, -1) {$2^{\prime \prime}$};
\node[] at (8.4, -2) {$3^{\prime \prime}$};
\node[] at (8.4, -3) {$4^{\prime \prime}$};
\node[] at (8.4, -4) {$5^{\prime \prime}$};
\foreach \x in {0, 0, 0, 0} \foreach \y in {0, -1, -2, -3, -4} \filldraw (\x,\y) circle (3pt);
\foreach \x in {4, 4, 4} \foreach \y in {0, -1, -2, -3} \filldraw (\x,\y) circle (3pt);
\filldraw (4, -4) [thick, gray] circle (3pt);
\foreach \x in {8, 8, 8, 8} \foreach \y in {0, -1, -2, -3, -4} \filldraw (\x,\y) circle (3pt);
\draw (0, -4) -- (4, 0);
\draw (4, 0) .. controls (3.6, -1) ..  (4, -2);
\draw (0, 0) .. controls (0.3, -0.5) ..  (0, -1);
\draw (0, -1) .. controls (0.3, -1.5) ..  (0, -2);
\draw (0, -2) .. controls (0.3, -2.5) ..  (0, -3);
\draw (4, 0) -- (8, 0);
\draw (8, 0) .. controls (7.7, -0.5) ..  (8, -1);
\draw (4, -2) -- (8, -2);
\draw (4, -1) .. controls (4.3, -1.5) ..  (4, -2);
\draw (4, -3) -- (8, -3);
\end{tikzpicture}
\end{center}
and we see that $\langle\pi_1,\pi_2\rangle=1$ as only the singleton light-grey block consists of elements with single prime. Hence
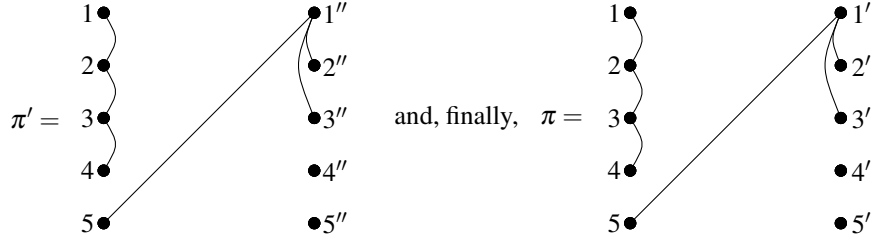
\begin{center}
\begin{tikzpicture}
[scale=0.7]
\node[] at (-2.3, -2) {$\pi^\prime = $};
\node[] at (-1.3, 0) {$1$};
\node[] at (-1.3, -1) {$2$};
\node[] at (-1.3, -2) {$3$};
\node[] at (-1.3, -3) {$4$};
\node[] at (-1.3, -4) {$5$};
\node[] at (3.4, 0) {$1^{\prime \prime}$};
\node[] at (3.4, -1) {$2^{\prime \prime}$};
\node[] at (3.4, -2) {$3^{\prime \prime}$};
\node[] at (3.4, -3) {$4^{\prime \prime}$};
\node[] at (3.4, -4) {$5^{\prime \prime}$};
\foreach \x in {-1, -1, -1, -1} \foreach \y in {0, -1, -2, -3, -4} \filldraw (\x,\y) circle (3pt);
\foreach \x in {3, 3, 3, 3} \foreach \y in {0, -1, -2, -3, -4} \filldraw (\x,\y) circle (3pt);
\draw (-1, -4) -- (3, 0);
\draw (3, 0) .. controls (2.6, -1) ..  (3, -2);
\draw (3, 0) .. controls (2.8, -0.5) ..  (3, -1);
\draw (-1, 0) .. controls (-0.7, -0.5) ..  (-1, -1);
\draw (-1, -1) .. controls (-0.7, -1.5) ..  (-1, -2);
\draw (-1, -2) .. controls (-0.7, -2.5) ..  (-1, -3);
\node[] at (5.7, -2) {and, finally,};
\node[] at (7.7, -2) {$\pi = $};
\node[] at (8.7, 0) {$1$};
\node[] at (8.7, -1) {$2$};
\node[] at (8.7, -2) {$3$};
\node[] at (8.7, -3) {$4$};
\node[] at (8.7, -4) {$5$};
\node[] at (13.4, 0) {$1^\prime$};
\node[] at (13.4, -1) {$2^\prime$};
\node[] at (13.4, -2) {$3^\prime$};
\node[] at (13.4, -3) {$4^\prime$};
\node[] at (13.4, -4) {$5^\prime$};
\foreach \x in {9, 9, 9, 9} \foreach \y in {0, -1, -2, -3, -4} \filldraw (\x,\y) circle (3pt);
\foreach \x in {13, 13, 13, 13} \foreach \y in {0, -1, -2, -3, -4} \filldraw (\x,\y) circle (3pt);
\draw (9, -4) -- (13, 0);
\draw (13, 0) .. controls (12.6, -1) ..  (13, -2);
\draw (13, 0) .. controls (12.8, -0.5) ..  (13, -1);
\draw (9, 0) .. controls (9.3, -0.5) ..  (9, -1);
\draw (9, -1) .. controls (9.3, -1.5) ..  (9, -2);
\draw (9, -2) .. controls (9.3, -2.5) ..  (9, -3);
\end{tikzpicture}
\end{center}
We conclude that $(\pi_1;s_1)(\pi_2;s_2)=(\pi;s_1+s_2+1)$.

The above defined multiplication in $\mathcal{P}^\tau_n$ is associative \cite{Martin94} and its restriction to $\mathcal{B}^\tau_n$ coincides with the multiplication in the twisted Brauer monoid defined in Section~\ref{subsec:twisted}. Hence,  $\mathcal{B}^\tau_n$ is a submonoid in $\mathcal{P}^\tau_n$, and it is easy to see that the identity element of $\mathcal{B}^\tau_n$ serves as the identity element for $\mathcal{P}^\tau_n$ as well. Thus, $\mathcal{P}^\tau_n$ is a monoid called the \emph{twisted partition monoid}.

The machinery developed in the present paper works, with minor adjustments, for twisted partition monoids and yield the following analogue of Theorem~\ref{thm:main}:
\begin{theorem}
\label{thm:partitions}
For each $n\ge 5$, the problem \textsc{Check-Id}$(\mathcal{P}^\tau_n)$ is co-NP-complete.
\end{theorem}
The detailed proof of Theorem~\ref{thm:partitions} will be published elsewhere.

Similar results can be obtained for various submonoids of $\mathcal{P}^\tau_n$, provided that they share the structure of their maximal subgroups with $\mathcal{P}^\tau_n$ and $\mathcal{B}^\tau_n$. For instance, an analogue of Theorems~\ref{thm:main} and~\ref{thm:partitions} holds for the \emph{twisted partial Brauer monoid}, that is, the submonoid of  $\mathcal{P}^\tau_n$ formed by all pairs $(\pi;s)$ such that each block of the partition $\pi$ consists of \emph{at most} two elements (as opposed to \emph{exactly} two elements for the case of $\mathcal{B}^\tau_n$).

\subsection{Open questions} We have already mentioned that the question of the complexity of identity checking in the monoid $\mathcal{B}^\tau_3$ is left open. Another interesting question concerns the complexity of identity checking in the Kauffman monoids $\mathcal{K}_n$ with $n\ge 5$ (For $n\le 4$, the problem \textsc{Check-Id}$(\mathcal{K}_n)$ is known to be polynomial time decidable; see~\cite{Chen20,KV20}.) The approach of the present paper does not apply to Kauffman monoids since their subgroups are trivial. So, some fresh ideas are needed to handle this case.

Another natural family of submonoids in twisted Brauer monoids is the twisted version of Jones's \emph{annular monoids} \cite{Jones94annular}. This version comes from the representation of partitions of $[n]\cup[n]'$ by annular rather than rectangular diagrams. Map the elements of $[n]$ to the $n$-th roots of unity doubled and the elements of $[n]'$ to the $n$-th roots of unity:
\[
k\mapsto 2e^{\frac{2\pi i(k-1)}{n}}\ \text{ and }\ k'\mapsto e^{\frac{2\pi i(k-1)}{n}}\ \text{ for all }\ k\in [n].
\]
Then the wires of any partition of $[n]\cup[n]'$ can be drawn in the complex plane  as lines within the annulus $\{z\mid 1<|z|<2\}$ (except for their endpoints).
For example, the annular diagram in Fig.\,\ref{fig:annular diagram} (taken from \cite{ADV12}) represents the partition $\{\{1,1'\},\{2,4\},\{3,2'\},\{3',4'\}\}$.
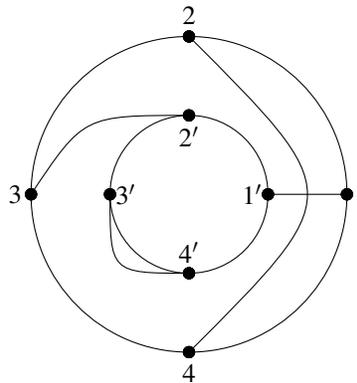
\begin{figure}[ht]
\centering
\begin{tikzpicture}
[scale=0.7]
\draw  (0,0) circle(1.5) circle(3);
\foreach \x in {-3, -1.5, 1.5, 3} \foreach \y in {0, 0, 0, 0} \filldraw (\x,\y) circle (3pt);
\foreach \x in {0, 0, 0, 0} \foreach \y in {-3, -1.5, 1.5, 3} \filldraw (\x,\y) circle (3pt);
\draw (1.5, 0) -- (3, 0);
\draw (0, 3) .. controls (3, 0) ..  (0, -3);
\draw (-3, 0) .. controls (-2, 1.5) ..  (0, 1.5);
\draw (-1.5, 0) .. controls (-1.5, -1.5) ..  (0, -1.5);
\node[] at (1.2, 0) {$1^\prime$};
\node[] at (3.3, 0) {$1$};
\node[] at (-1.2,0) {$3^\prime$};
\node[] at (-3.3,0) {$3$};
\node[] at (0, 1.1) {$2^\prime$};
\node[] at (0, 3.4) {$2$};
\node[] at (0, -1.1) {$4^\prime$};
\node[] at (0, -3.4) {$4$};
\end{tikzpicture}
\caption{Annular diagram of a partition of $[4]\cup[4]'$}\label{fig:annular diagram}
\end{figure}

The \emph{twisted annular monoid} $\mA^\tau_n$ is the submonoid of $\mathcal{B}^\tau_n$ consisting of all elements whose partitions have a representation as an annular diagram whose wires do not cross.
The subgroups of $\mA^\tau_n$ are known to be finite and cyclic, and checking identities in finite cyclic groups is easy: an identity $w\bumpeq w'$ holds in the cyclic group of order $m$ if and only if for every letter in $\alf(ww')$, the number of its occurrences in $w$ is congruent modulo $m$ to the number of its occurrences in $w'$. Thus, the problem \textsc{Check-Id}$(\mathcal{A}^\tau_n)$ also cannot be handled with the approach of the present paper and needs new tools.

\medskip

\paragraph{\emph{Acknowledgement}.} The authors are extremely grateful to the anonymous referee who spotted several inaccuracies and suggested shorter alternative proofs of Corollary~\ref{cor:stable} and Proposition~\ref{prop:subgroups} incorporated in the present version.

\end{document}